\pdfoutput=1

\documentclass[a4paper]{amsart}

\usepackage{graphicx}
\usepackage{epstopdf}
\usepackage{subcaption}
\usepackage{lipsum} 
\usepackage{amsrefs}
\usepackage{sidecap}
\usepackage{float}% to make \paragraph display in boldface, use this:
\makeatletter
\renewcommand\paragraph{\@startsection{paragraph}{4}%
  \z@ \z@ {-\fontdimen2\font}%
  {\normalfont\bfseries}} % <-- this sets it to bold
\makeatother
\usepackage{tgtermes}  %FONT STYLE Times

\usepackage{color}
\usepackage{mathrsfs}
\usepackage{tikz-cd}
\usepackage{stmaryrd}

\usepackage{enumitem}
\newlist{myQuoteEnumerate}{enumerate}{2}% Set max nesting depth
\setlist[myQuoteEnumerate,1]{label=(\alph*)}% Use numbers for level 1
\setlist[myQuoteEnumerate,2]{label=(\alph*)}%   Use letters for level 2

\newcommand{\gap}{\vspace{0.2cm}}

\usepackage{amsmath,amssymb,latexsym}
\usepackage[utf8]{inputenc}
\usepackage[english]{babel}
\usepackage{booktabs}
\newtheorem{thm}{Theorem}[section]
\newtheorem{thmm}{Theorem}[section]
\newtheorem{cor}{Corollary}[section]
\newtheorem{lem}{Lemma}[section]

\newtheorem{defi}{Definition}[section]
\newtheorem{example}{Example}[section]
\newtheorem{remark}{Remark}[section]

\usepackage{siunitx}
\usepackage{xcolor}
\usepackage{booktabs,colortbl, array}
\usepackage{pgfplotstable}

\definecolor{rulecolor}{RGB}{0,71,171}
\definecolor{tableheadcolor}{gray}{0.92}

\numberwithin{equation}{section}

\newcommand\quotient[2]{
        \mathchoice
            {% \displaystyle
                \text{\raise1ex\hbox{$#1$}\Big/\lower1ex\hbox{$#2$}}%
            }
            {% \textstyle
                #1\,/\,#2
            }
            {% \scriptstyle
                #1\,/\,#2
            }
            {% \scriptscriptstyle  
                #1\,/\,#2
            }
    }

\definecolor{aurometalsaurus}{rgb}{0.43, 0.5, 0.5}
\definecolor{darkjunglegreen}{rgb}{0.1, 0.14, 0.13}
\definecolor{coolblack}{rgb}{0.0, 0.18, 0.39}
\definecolor{cobalt}{rgb}{0.0, 0.28, 0.67}

\usepackage{hyperref}
\hypersetup{
     colorlinks   = true,
     citecolor    = blue
}

\title[Subhyperbolicity at Misiurewicz points]{Misiurewicz points and subhyperbolicity in unicritical algebraic correspondences}

\author{Carlos Siqueira }

\date{\today}

\address{Department of Mathematics, Institute of Mathematics and Statistics, 
Federal University of Bahia, Salvador -- BA, Brazil}
\curraddr{Department of Mathematics, Institute of Mathematical and Computer Sciences, 
University of São Paulo, São Carlos -- SP, Brazil}
\email{carlos.siqueira@ufba.br}

\begin{document}

\hypersetup{linkcolor=cobalt}

\begin{abstract}

We provide the first definition of \emph{Misiurewicz parameter} for the unicritical family of algebraic correspondences \( z^r + c \), with \( r > 1 \) rational, and prove that, at every Misiurewicz parameter, the correspondence uniformly expands the canonical orbifold metric on a neighborhood of the Julia set. This is achieved using Thurston's ideas on postcritically finite rational maps, regular branched coverings, and orbifolds, viewing the correspondence as a global analytic multifunction. 

This result provides the necessary tools for further investigations into the fine structure of the parameter space near Misiurewicz points, particularly in exploring similarities between the local geometry of the parameter space and the Julia sets at such parameters. 
Finally, we present both rigorous examples and empirical evidence suggesting that Misiurewicz parameters are abundant and may be detected by identifying increasingly small copies of the Multibrot set nested within itself: the smaller the copy, the closer it is likely to be to a Misiurewicz parameter.

\end{abstract}

\maketitle


\tableofcontents

\section{Introduction}\label{asdfasdfbasdfcccds}

In the early 1980s, Douady and Hubbard, drawing on Thurston's insights into the topological characterization of postcritically finite rational maps~\cite{DH93}, gave the first definition of subhyperbolic maps in~\cite{DH84}, providing characterizations based on the behavior of critical orbits.  
In that work, they also carried out a detailed study of polynomial maps of the form $z^d + c$ for degrees $d > 1$.

\gap
\paragraph{Main results.} Building on Thurston’s orbifold theory~\cite{DH93} and elementary sheaf theory, our main contributions are Theorems~\ref{gdsadgwed}, \ref{ljksoijdfkwdfsed}, \ref{alkjasdpoipoallskdf}, and~\ref{lkjajhlkjhasdfoiqdfadf}, presented in the next section, where we extend the notion of subhyperbolicity to algebraic correspondences of the form~\eqref{lkjdllkjhalsoed}.
 This setting naturally generalizes the classical family $f_c(z) = z^d + c$ to $\mathbf{f}_c(z) = z^r + c$, where $r > 1$ is a rational exponent. Our main result establishes that if $a \in \mathbb{C}$ is a Misiurewicz point (Definition \ref{dlkhaskjhdfjhkwefdc}), then $\mathbf{f}_c$ expands a suitable orbifold metric uniformly on a neighborhood of the Julia set, with only finitely many singularities. As a consequence, we obtain that the filled Julia set coincides with the Julia set
 at any such Misiurewicz point~$a$.
This more general framework, rooted in the theory of algebraic correspondences, has been the subject of significant attention in recent decades. 
\gap
\paragraph{Key developments on algebraic correspondences}
\label{lljasdflhlhbdqefsdf}
The advent of computer graphics in the late 1970s, together with Mandelbrot's iconic visualizations of fractals, sparked a global interest in holomorphic dynamics. This revival also brought renewed attention to multi-valued systems (long studied in the context of Fuchsian groups), particularly following Bullett's work in the early 1990s on critically finite correspondences \cite{Bullett92}, and the subsequent breakthrough by Bullett and Penrose, who introduced a family~$\mathcal{F}_a$ of algebraic correspondences  (matings) whose connectedness locus was conjectured to be homeomorphic to the Mandelbrot set~\cite{Bullett1994}.

Nearly three decades later, this conjecture was confirmed through a series of major advances. Bullett and Lomonaco~\cite{BL19, BULLETT2024109956} provided a rigorous proof of the conjecture originally posed in~\cite{Bullett1994}, while independent developments by Lee, Lyubich, Makarov, and Mukherjee~\cite{Mukherjee} introduced a family~$\mathcal{S}$ of Schwarz reflections that yield anti-holomorphic correspondences, realized as matings between anti-rational maps and the abstract modular group. One of the most striking results in~\cite{Mukherjee} is the construction of a homeomorphism between the abstract connectedness locus of $\mathcal{S}$  and the abstract parabolic Tricorn, the combinatorial model of the Tricorn. The latter can be interpreted as the connectedness locus of a family of anti-holomorphic quadratic polynomials, also known as the anti-holomorphic Mandelbrot set.

\gap
\paragraph{Application to asymptotic similarity.} Our results serve as key tools for analyzing the similarity between Multibrot sets and Julia sets arising from the dynamics of $\mathbf{f}_c$. Just as Tan Lei  \cite{TanLei} established an analogous  phenomenon for the quadratic family $z^2 + c$, building on earlier results developed by Douady and Hubbard~\cite{DH84}, we will apply the theorems of this paper to extend Tan Lei's results to correspondences $\mathbf{f}_c(z) = z^r + c$, where $r > 1$ is rational. This generalization is carried out in detail in~\cite{siqueira2025quadratic}. See also Figure~\ref{flhalkjhsdfopaoiusdfad} for an illustration in this context.

 \gap \paragraph{On critically finite correspondences.}
It is worth noting that the application of Thurston's classification to the setting of algebraic correspondences was first explored in the seminal work of Bullett~\cite{Bullett92}, who classified critically finite quadratic correspondences (those in which every critical point has a finite full orbit under both forward and backward iteration), showing that such systems exhibit strong rigidity. In particular, Bullett demonstrated that, up to conformal conjugacy, there exist only eleven such correspondences of quadratic type.

\section{Definitions and main results}\label{adfasdflkhblkjhsoidf}

 Recall that for the quadratic family $f_c(z) = z^2 + c$, a point $c \in \mathbb{C}$ is called \emph{Misiurewicz} if the critical point is strictly preperiodic under iterations by  $f_c.$ Thanks to special geometric properties of $f_c$, it can be shown that the critical point eventually maps to a repelling cycle when the parameter $c$ is a Misiurewicz point. 
As previously described in~\cite{ETDS22, SS17, Proc22}, many of the geometric features characteristic of the quadratic family -- such as rigidity, hyperbolic Julia sets of zero area, and  holomorphic motions -- also extend to maps of the form
\(
 z^r + c,
\)
where $r=p/q > 1$ is a rational number. However, such maps are no longer single-valued, but rather define algebraic correspondences given by
\begin{equation}\label{lkjdllkjhalsoed}
\mathbf{f}_c(z) = \left\{ w \in \mathbb{C} : (w - c)^q = z^p \right\},
\end{equation}
where $p > q$ are integers in $[2, \infty)$.

\textcolor{black}{See Remark~\ref{asldkjsspsdflsdjf} for the motivation behind the following definition.}

\begin{defi}[\bf Misiurewicz point]\label{dlkhaskjhdfjhkwefdc}\normalfont A parameter $a$ of the family of holomorphic correspondences $\mathbf{f}_c$ is a \emph{Misiurewicz point} if  $(i)$ the critical point $0$ has only one bounded forward  orbit \[ 0 \xrightarrow{\mathbf{f}_a} z_0 =a\xrightarrow{\mathbf{f}_a} z_1 \xrightarrow{\mathbf{f}_a} \cdots \] and $(ii)$ this orbit is strictly preperiodic. The first point of this orbit which is periodic is denoted   by $z_{\ell}$ and the associated cycle $\alpha_a$ of period $n$ is \[z_{\ell} \mapsto z_{\ell +1}  \mapsto \cdots \mapsto z_{\ell +n}=z_{\ell}.\]  
\end{defi}

Repelling cycles are defined in Section 
\ref{gdasdfasdc}. Theorem~\ref{lkjajhlkjhasdfoiqdfadf} shows that the cycle $\alpha_a$ is repelling. 

{\color{black}The motivation for the following definition is discussed in Remark~\ref{asdfakjhsoidfsdf}.}

\begin{defi}\normalfont
The \emph{filled Julia set} $K_c$ is the set of all points $z$ in the complex plane that have at least one bounded forward orbit.
\end{defi}
Next, we introduce a generalization of the Mandelbrot set corresponding to the family~\eqref{lkjdllkjhalsoed}.

\begin{defi}[\bf Multibrot set] \label{ahlkjhasdfpooadsfa}
\normalfont
Let \( p, q \) be positive integers with \( p > q \geq  1 \). The \emph{Multibrot set} \( M_{p,q} \) associated with the family~\eqref{lkjdllkjhalsoed} is defined by
\[
M_{p,q} = \{ c \in \mathbb{C} \mid 0 \in K_c \},
\]
where \( K_c \) denotes the filled Julia set of the correspondence~\eqref{lkjdllkjhalsoed} for the given pair \( (p, q) \).
\end{defi}

For the quadratic family \( f_c(z) = z^2 + c \), Misiurewicz points form a countable and dense subset of the boundary of the Mandelbrot set; see~\cite{milnor1989self}.

This naturally raises the question of whether a similar density of Misiurewicz points holds on the boundary of Multibrot sets arising from holomorphic correspondences. While we do not address this difficult problem in the present paper, we do provide a first example illustrating the existence of infinitely many Misiurewicz points for the family \( \mathbf{f}_c \) with \( p = 4 \) and \( q = 2 \). In this case, the dynamics of \( \mathbf{f}_c \) corresponds to the semigroup generated by the pair \( \langle z^2 + c,\ -z^2 + c \rangle \).

\begin{example}\label{asdgalkhcjsdfs}  
\normalfont
We will prove that there exist infinitely many Misiurewicz points in a neighborhood of \( c = -2 \) for the semigroup family \( \langle z^2 + c,\ -z^2 + c \rangle \). The parameter \( c = -2 \)  is a Misiurewicz point for this family: we will show that the critical point \( 0 \) has only one bounded forward orbit, which is strictly preperiodic, mapping to the fixed point \( 2.\)

 \begin{figure}[h]
  \centering
  \begin{subfigure}[b]{0.5\linewidth}
    \centering
    \includegraphics[width=\linewidth]{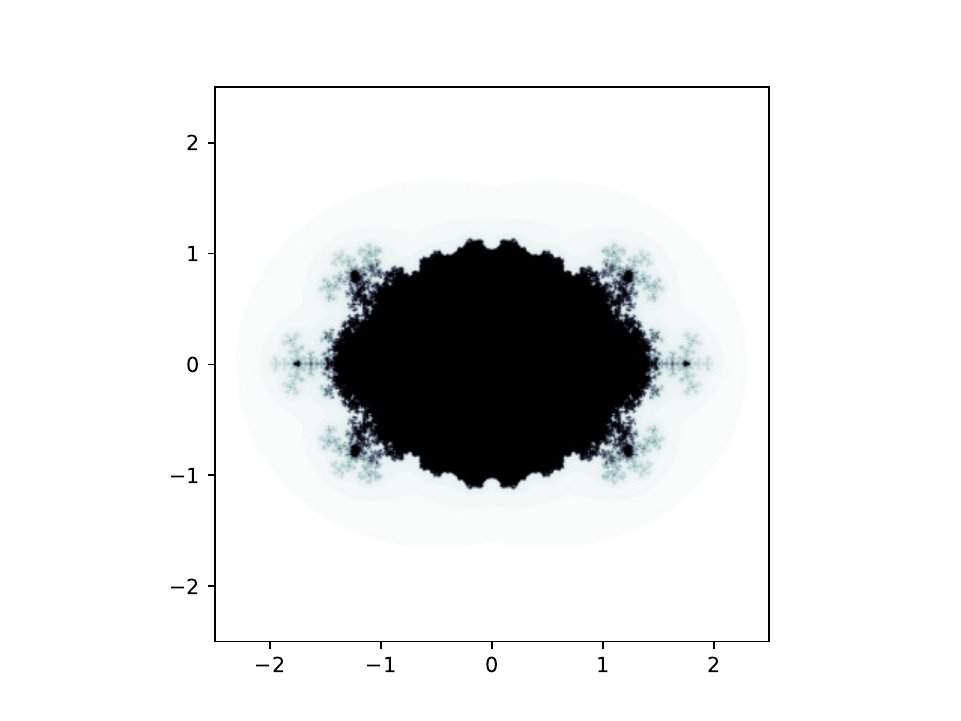}
    \caption{{\tiny $M_{4,2}$}} 
    \label{lkhasdgopohasdg}
  \end{subfigure}%
  \hfill
   \begin{subfigure}[b]{0.5\linewidth}
    \centering
    \includegraphics[width=\linewidth]{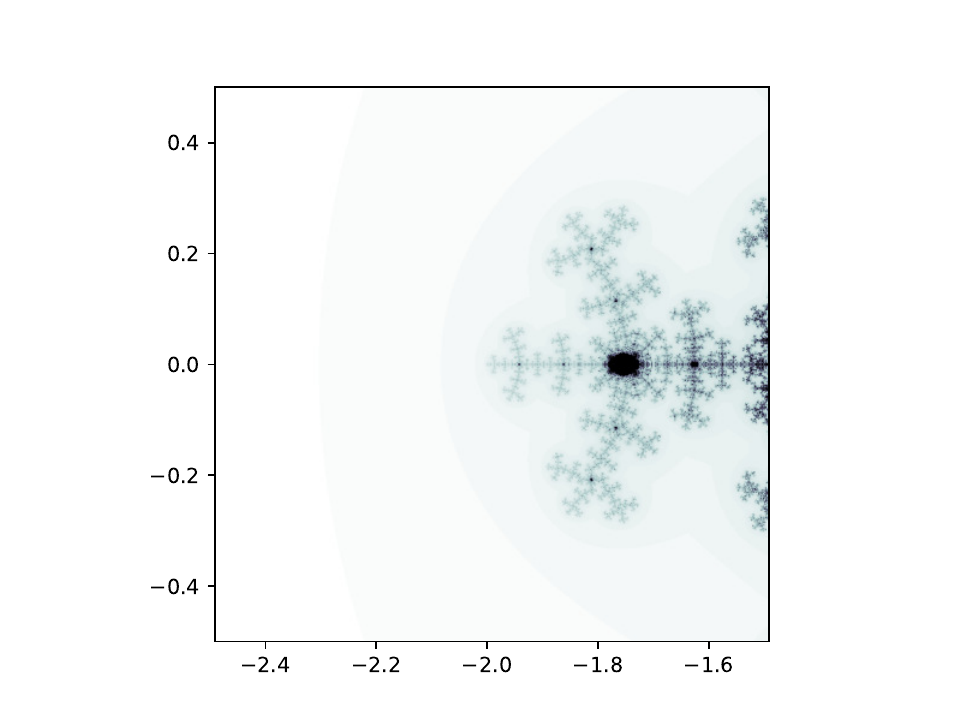}
    \caption{{\tiny $M_{4,2}$ zoomed $10\times$ near $ -2$
}}
    \label{lkhbsljhssdfasdflkkcs}
  \end{subfigure}%
  \hfill
   \begin{subfigure}[b]{0.5\linewidth}
    \centering
    \includegraphics[width=\linewidth]{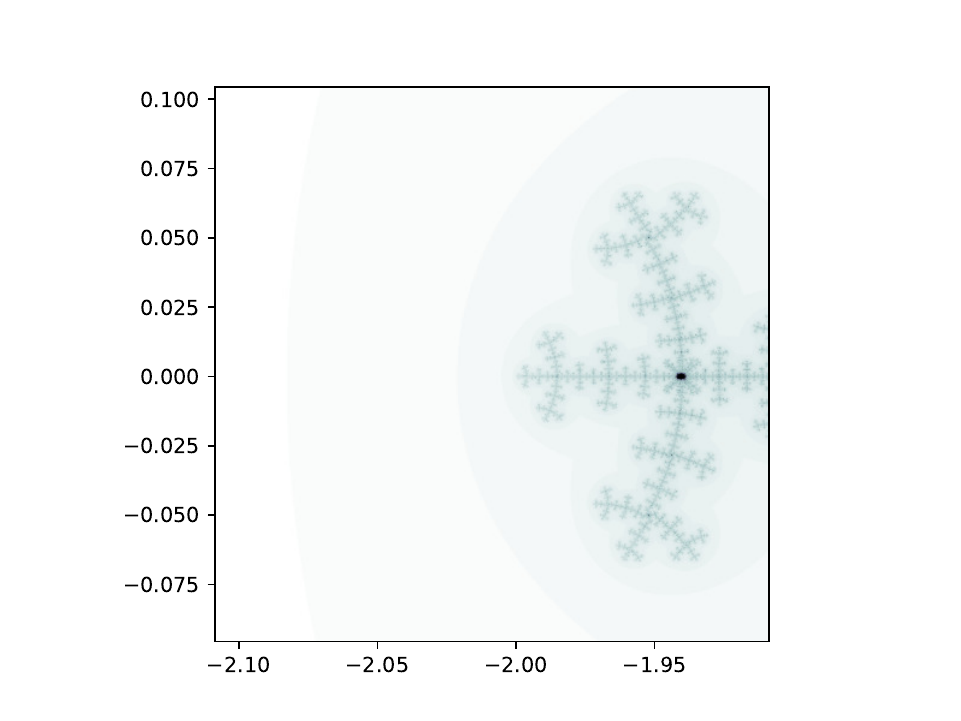}
    \caption{\tiny $M_{4,2}$ zoomed $100\times$ near $ -2$}
    \label{lhlkjhapplkspoqweprqc}
  \end{subfigure}%
  \hfill
   \begin{subfigure}[b]{0.5\linewidth}
    \centering
    \includegraphics[width=\linewidth]{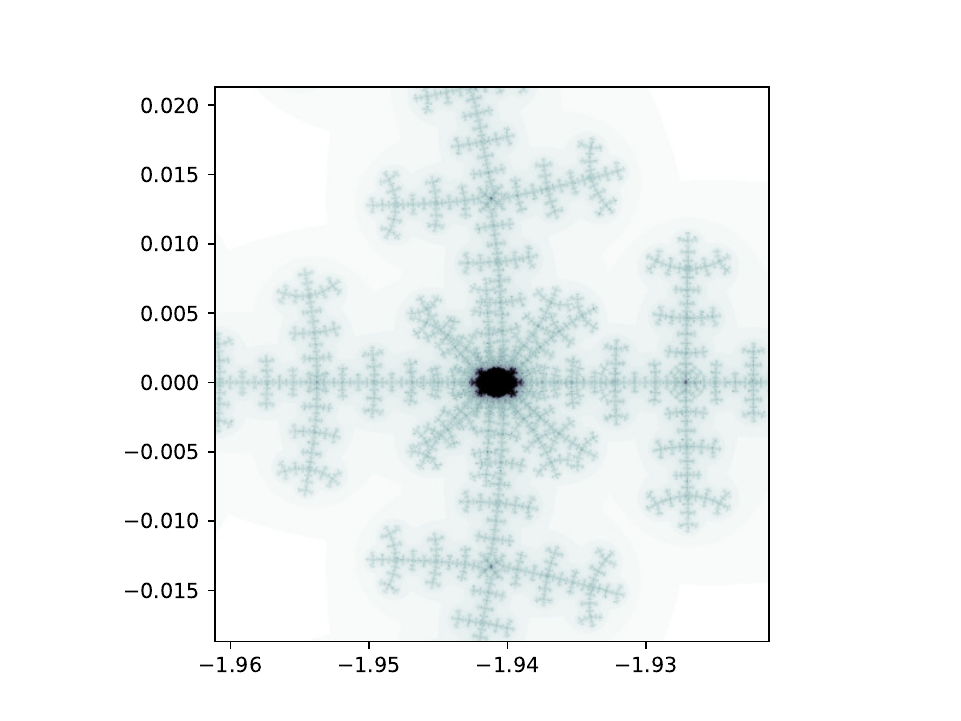}
    \caption{\tiny $M_{4,2}$ zoomed $1000\times$ near $ -1.94$}
    \label{cpoiapopqadfacbqew}
  \end{subfigure}%
  \hfill
   \begin{subfigure}[b]{0.5\linewidth}
    \centering
    \includegraphics[width=\linewidth]{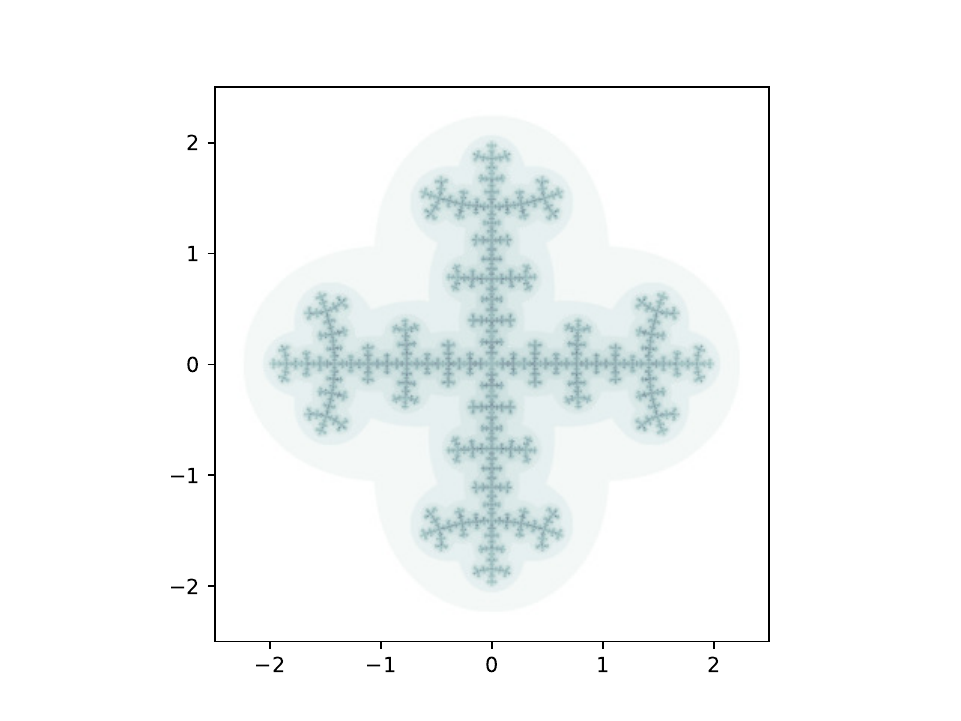}
    \caption{\tiny $K_c,$ $c=-2,$ empty interior}
    \label{poiapoiusfapoibdadqwefd}
    \end{subfigure}%
     \hfill
   \begin{subfigure}[b]{0.5\linewidth}
    \centering
    \includegraphics[width=\linewidth]{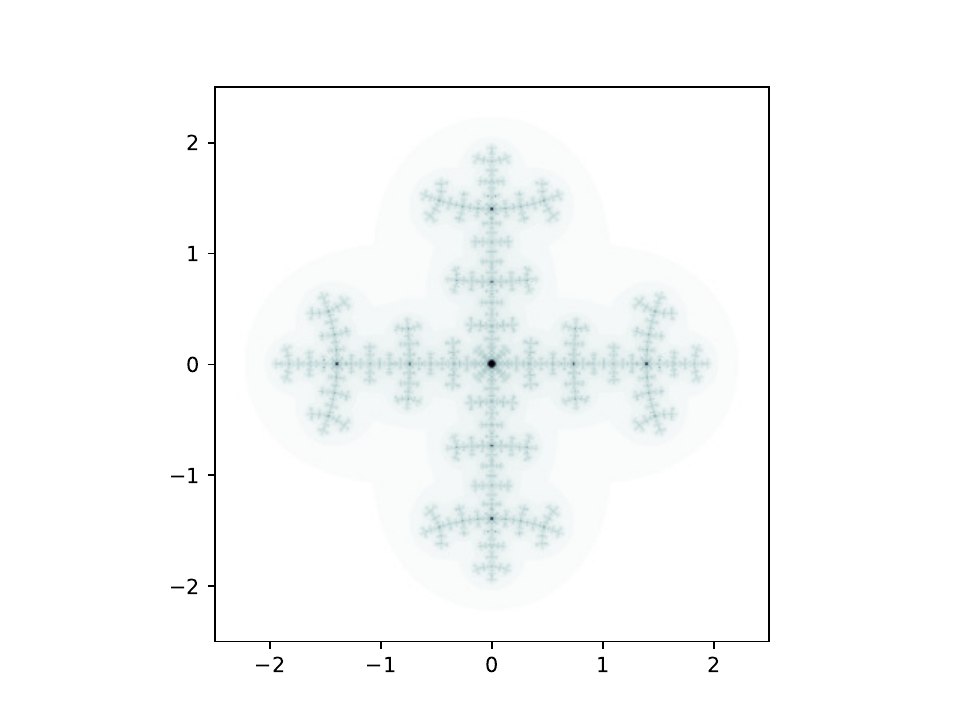}
    \caption{\tiny $K_c,$ $c=-1.94,$ nonempty interior}
    \label{poiapoiussdfsdfqwefd}
    \end{subfigure}%
   
  \caption{ The Multibrot set $M_{4,2}$ and related copies, illustrating the similarity between $M_{4,2}$ and the filled Julia set $K_{-2}$ near the parameter $-2$. There exists an infinite sequence of Misiurewicz points $c_n$, each located near a small copy of $M_{4,2}$, and converging to $-2$. See Example~\ref{asdgalkhcjsdfs}. The point $-2$ is not exceptional; similar patterns occur near infinitely many copies of $M_{4,2}$, as well as in other Multibrot sets $M_{p,q}$, as discussed in~\cite{siqueira2025quadratic}.
 }
  \label{flhalkjhsdfopaoiusdfad}
\end{figure}

Indeed,  a  direct computation shows that the only bounded forward orbit of the critical point under \( \mathbf{f}_{-2} \) is 
\[
0 \mapsto -2 \mapsto 2 \mapsto 2.
\] 
Thus \( c = -2 \) is a Misiurewicz point for the correspondence. Moreover, the intersection of \( \mathbf{f}_{-2}^2(0) \) with the open ball \( B(0,6) = \{ z \in \mathbb{C} : |z| < 6 \} \) is precisely \( \{2\} \). 

Since the complement of this ball is forward invariant and lies in the basin of infinity for all \( c \) sufficiently close to $-2$, it follows by stability that \( \mathbf{f}_c^2(0) \) intersects \( B(0,6) \) in a single point. Thus when \( c \) is close to \( -2 \), there is at most one bounded forward orbit of the critical point under \( \mathbf{f}_c \).

On the other hand, Misiurewicz parameters are dense in the boundary \( \partial M \) of the Mandelbrot set for the quadratic family \( z^2 + c \), so there exists a sequence \( c_n \to -2 \) for which the orbit of \( 0 \) under \( z^2 + c_n \) is strictly preperiodic. Since this orbit is also admissible under the correspondence, each \( c_n \) admits exactly one bounded, strictly preperiodic orbit of the critical point under \( \mathbf{f}_{c_n} \). Hence each \( c_n \) is a Misiurewicz point for the correspondence.

In contrast, although \( c = i \) is a Misiurewicz parameter for the quadratic family, it fails to satisfy the same property for the correspondence, as the critical point has multiple bounded forward orbits and is not preperiodic to a single repelling cycle. Indeed, after computing the iterate \( \mathbf{f}_i^3(i) \), we find that only four orbits of \( 0 \) remain bounded, while all others escape to the basin of infinity. These four distinct orbits eventually land on four different  cycles, so the critical point is not preperiodic to a single cycle. Thus \( i \) is not a Misiurewicz parameter for the correspondence.
\end{example}

\begin{remark}\normalfont
It should be noted that, according to the results in \cite[\S 2.1]{ETDS22},  the filled Julia set $K_c$ is connected for every parameter $c \in M_{p,q}$. The set $M_{2,1}$ coincides with the classical Mandelbrot set.     \end{remark}
{\color{black}
\begin{remark}\label{asldkjsspsdflsdjf}
\normalfont
Consider the quadratic family $f_c(z)=z^2+c$. If $a\in\mathbb{C}$ is a Misiurewicz parameter, it is well known that:
\begin{enumerate}
\item $J_a = K_a$, and the filled Julia set $K_a$ has empty interior;
\item there exist infinitely many small copies of the Mandelbrot set accumulating at $a$;
\item the Mandelbrot set $M$ and the filled Julia set $K_a$ are asymptotically similar in a neighborhood of $a$.
\end{enumerate}

As shown in Example~\ref{asdgalkhcjsdfs}, when $c=i$ the critical point is preperiodic to multiple distinct cycles under the dynamics of the semigroup $\langle z^2+c,\,-z^2+c\rangle$. In this situation,  experiments indicate that the associated filled Julia set satisfies none of the properties~$(1)$--$(3).$
These observations motivate our definition of a Misiurewicz point. Allowing the critical point to be preperiodic to multiple cycles may lead to interesting phenomena; however, such behavior does not generalize the classical polynomial case. 
\end{remark}}

The following theory, due to Thurston~\cite{DH93}, is both elegant and deserving of a clear presentation, as it underpins our main results.

\gap \paragraph{Branched coverings and orbifolds}
An \emph{orbifold} is a pair $(X, \nu)$ where $X$ is a Riemann surface and \textcolor{black}{$\nu: X \to \{ 1, 2, \ldots \}$} is a function such that \emph{the set of ramified points} $\{ x \in X: \nu(x) >1\} $ is locally finite. The integer $\nu(x)$ is the \emph{branch index} of $x.$

Suppose that $X$ and $Y$ are Riemann surfaces. A surjective holomorphic map $\varphi: X \to Y$ is \emph{proper} if the inverse image of any compact set in $Y$ is a compact set in $X.$ Proper maps satisfy some interesting properties: the set of branch points $B=\{ x\in X: \varphi'(x) =0\}$ is locally finite, as well as $R=\varphi(B)$ and $\varphi^{-1}(R)$. The induced map $f: X {\setminus} \varphi^{-1}(R) \to Y {\setminus} R$ is a covering map of finite degree $d$. For this reason proper maps are also known as \emph{$d$-fold branched coverings.}

A surjective holomorphic map $\varphi: X \to Y$ is a \emph{branched covering} if every $y\in Y$ has a neighborhood $U$ such that $\varphi$ maps every connected component of $\varphi^{-1}(U)$ onto $U$ as a proper map. If there exists a subgroup $\Gamma$ of the group of conformal automorphisms of $X$ such that $\varphi(x_1) =\varphi(x_2)$ if, and only if, $\gamma.x_1=x_2$ for some $\gamma \in \Gamma$, then $\Gamma$ is uniquely determined and  $\varphi$ is called a \emph{regular branched covering.} This is the well-known \emph{group of deck transformations}.

A few properties of regular branched coverings: the set of branch points  $B$ is locally finite and $\Gamma.B=B.$   The local degree of $\varphi$ at a point $x \in B$ is $n$ if $\varphi^{(k)}(x) = 0$ for all $1 \leq k < n$, and $\varphi^{(n)}(x) \neq 0$.
The local degree of $\varphi$ at  $x \in B$ is the same as the local degree of $\varphi$ at $\gamma.x,$ for all $\gamma \in \Gamma.$ Any  point of $R=\varphi(B)$ is a \emph{ramified point.}  Since $\varphi$ is regular, $\varphi^{-1}(R)=B;$ moreover,  the local degree of $\varphi$ at every point in the inverse image of a ramified point $y$ is the same integer $d$, defined as the \emph{ramification index} $\nu(y)=d.$ This defines a \emph{ramification function} $\nu : Y \to \{1, 2, \ldots\}$ associated with the regular branched covering $\varphi$.
 Notice that $y$ is a ramified point if, and only if, $\nu(y) > 1.$

\begin{thm}[\bf Thurston, Douady and Hubbard]\label{gjcgdwoood} Let $(S, \nu)$ be a Riemann surface orbifold that is conformally isomorphic to the complex plane.  Suppose that $\nu$ is non-trivial, that is, we have at least two ramified points with different ramification indices. Then there exists a regular branched covering $\varphi: \tilde{S}_{\nu} \to (\mathbb{C}, \nu)$, unique up to conformal isomorphisms,  such that $\tilde{S}_{\nu}$ is conformally isomorphic to the hyperbolic disk $\mathbb{D}$  and the ramification function of this covering is the given $\nu.$

\end{thm}

\begin{proof}[Sketch of Proof and References]
Our statement follows the exposition of John Milnor in \cite{Milnor}, which itself is based on the original theory developed in the seminal paper \cite{DH93}. The result discussed here is a special case of Theorem E.1 in \cite{Milnor}, with the additional conclusion that the universal covering space is $\tilde{S}_{\nu} \simeq \mathbb{D}$, a fact we now proceed to justify.

According to Douady and Hubbard \cite{DH93}, the \emph{Euler characteristic of the orbifold} $(S, \nu)$ is defined by
\[
\chi(S, \nu) = \chi(S) + \sum_{j} \left( \frac{1}{\nu(a_j)} - 1 \right),
\]
where the sum is taken over all ramified points $a_j$. Based on the assumptions on  $\nu$, we may suppose that $\nu(a_1) > 1$ and $\nu(a_2) > 2$. Since the topological Euler characteristic is $\chi(S) = 1$, a straightforward calculation using the formula above yields $\chi(S, \nu) \leq -\frac{1}{6}$. 

It then follows from Lemmas E.3 and E.4 in \cite{Milnor} that the universal cover $\tilde{S}_{\nu}$ is hyperbolic; that is, $\tilde{S}_{\nu} \simeq \mathbb{D}$.
\end{proof}

Let $(\mathbb{C}, \nu)$ be an orbifold with a non-trivial ramification function.
Let $\varphi$ denote the regular branched covering given by Theorem~\ref{gjcgdwoood}.

Let $\varphi$ denote the regular branched covering given by  Theorem \ref{gjcgdwoood}. As usual, the set of branch points is $B=\{x \in \tilde{S}_{\nu}: \varphi'(x) =0\}$ and its image under $\varphi$ is the set of ramified points, which we denote by $R.$  Since $\varphi$ is regular, $\varphi^{-1}(R) =B.$ Both $B$ and $R$ are locally finite.  The restriction $\varphi: \tilde{S}_{\nu}{\setminus} B \to \mathbb{C}{\setminus} R$ is a covering map. Let $\rho_1(z)|dz|$ denote the Poincar\'e metric on $\tilde{S}_{\nu} \simeq \mathbb{D}$.
   Every element of the associated  group $\Gamma$ of deck transformations  is an isometry with respect to $\rho_1(z)|dz|.$ Using this fact we are allowed to make the following definition. 

\begin{defi}[\bf Orbifold metric]\label{lkjkjhgaiusdedd} 
\normalfont 
There exists a unique conformal metric $\rho_2(z)|dz|$ on $\mathbb{C}{\setminus} R$ such that 
\[
\varphi: (\tilde{S}_{\nu}{\setminus} B, \rho_1) \longrightarrow (\mathbb{C}{\setminus} R, \rho_2)
\] 
is a local isometry. This metric, called the \emph{orbifold metric} of $(\mathbb{C},\nu)$, is independent of the particular choice of $\varphi$; in fact, any $\varphi$ given by Theorem~\ref{gjcgdwoood} produces the same metric on $\mathbb{C}{\setminus} R$.
\end{defi}

The orbifold metric blows up at each ramified point $a_j$: in some punctured neighborhood $U^*$ of $a_j,$ we have  $\rho_2(z) \to \infty$ as $z\to a_j.$ (This property can be verified on page~211 of~\cite{Milnor}. Note that the local branched covering 
$z = a_j + w^{\nu(a_j)}$ and $\varphi$ have the same ramification index at $a_j$.)

\begin{defi}[\bf Ramified points of the correspondence]\label{lkjlhdoiehfeddd}\normalfont
Let $a$ be a Misiurewicz point for the family $\mathbf{f}_c$ given by~\eqref{lkjdllkjhalsoed}. The set of \emph{ramified points} is \begin{equation}\label{lkjhaskkjshdfbc}R= \bigcup_{k\geq 0} \mathbf{f}_a^k(0)\end{equation} where $\mathbf{f}_a^0$ denotes the identity map.
 Since this set is countable, we will denote its elements by $a_j$, with  $a_0=0$ and $j \geq 0.$ \end{defi}

By Lemma \ref{alkjacbcbkhlpsdf}, the set $R$ in \eqref{lkjhaskkjshdfbc} is locally finite.

\begin{defi}[\bf Canonical orbifold]\label{lkjlsjodijpoiefd}\normalfont
The \emph{ramification function}  $\nu_a$ associated with the Misiurewicz point $a$ is defined by setting $\nu_a(0)= q,$ $\nu_a(a_j)=p$ for every nonzero point $a_j$ in  the postcritical set, and $\nu_a(z)=1$ elsewhere.  It is always assumed that $q\geq 2$ and $p>q.$ The pair $(\mathbb{C}, \nu_a)$ is the \emph{canonical orbifold} of the correspondence.  The \emph{canonical orbifold metric} of $(\mathbb{C}, \nu_a)$ is the one given in Definition \ref{lkjkjhgaiusdedd} with $\nu=\nu_a.$
\end{defi}

{\color{black}
\paragraph{\bf The sheaf $\mathfrak{S}_{U}$} The following theorem is central to this work. In order to state it precisely, we briefly recall some standard notions from the theory of global analytic functions. Further background and related discussion can be found in Section~\ref{kjhkjhsdoiwe}.

Let $U\subset\mathbb{C}$ be a domain. The set of all germs $(f,z)$ of holomorphic functions with base point $z \in U$, denoted by $\mathfrak{S}_U$,
 carries a natural topology, under which it becomes the \emph{sheaf of germs of analytic functions} on $U$. Each connected component of this space is a Riemann surface, and the canonical projection
\(
\pi \colon (f,z)\mapsto z
\)
is a local homeomorphism.

Given a germ $(f,z)$ and a continuous path $\gamma\colon[0,1]\to U$ with $\gamma(0)=z$, an \emph{analytic continuation} of $(f,z)$ along $\gamma$ is a continuous lift $\widetilde{\gamma}\colon[0,1]\to \mathfrak{S}_U$ such that $\widetilde{\gamma}(0)=(f,z)$ and
\[
\pi\circ\widetilde{\gamma}(t)=\gamma(t)\quad\text{for all }t\in[0,1].
\]

A \emph{multifunction} $\mathbf{f}$ is a set--valued mapping defined on a domain $U \subset \mathbb{C}$, assigning to each point $z \in U$ a subset $\mathbf{f}(z) \subset \mathbb{C}$.  
A \emph{branch} of $\mathbf{f}$ is a holomorphic map $f$ such that $f(z) \in \mathbf{f}(z)$ for every $z$ in the domain of $f$.  
We say that $\mathbf{f}$ is a \emph{holomorphic multifunction} if, for every $z \in U$ and every $w \in \mathbf{f}(z)$, there exists a branch $f$ of $\mathbf{f}$ locally defined at $z$ such that $f(z)=w$.

Given a holomorphic multifunction $\mathbf{f} \colon U \to \mathbb{C}$, we denote by $\mathfrak{S}(\mathbf{f})$ the space of all germs $\mathfrak{f}=(f,z)$ arising from local branches $f$ of $\mathbf{f}$.  
If $\mathfrak{S}(\mathbf{f})$ is a connected component of the sheaf of germs of analytic functions on $U$, then $\mathbf{f}$ is called a \emph{global analytic multifunction}.

}
\begin{thmm}[\bf Decomposition] \label{gdsadgwed} Suppose  that $a$ is a Misiurewicz point of the family $\mathbf{f}_c$ given by \eqref{lkjdllkjhalsoed}.  Let $\varphi: \mathbb{D} \to \mathbb{C}$ be the unique (up to conformal isomorphisms) regular branched covering of  $(\mathbb{C}, \nu_a)$ which has the given $\nu_a$ as ramification function.  Then $\mathbf{g}(z) = \varphi^{-1}\circ \mathbf{f}_a^{-1} \circ \varphi(z)$ is a holomorphic multifunction  from $\mathbb{D}$ to $\mathbb{D}$ which can be decomposed into a family of global analytic multifunctions $\mathbf{g}_{\alpha}: \mathbb{D} \to \mathbb{D}$ such that \begin{equation} \label{fdfsdfsdfsdf} \mathfrak{S}(\mathbf{g}) = \bigcup_{\alpha} \mathfrak{S}(\mathbf{g}_{\alpha}) \textrm{\ and \ } \mathbf{g}(z)= \bigcup_{\alpha} \mathbf{g}_{\alpha}(z) \textrm{,\ for any $z\in \mathbb{D}.$}\end{equation}

\noindent Every germ $\mathfrak{f}$ in the Riemann surface $\mathfrak{S}(\mathbf{g}_{\alpha})$ can be continued along any curve in $\mathbb{D}$  starting at $\pi(\mathfrak{f}).$

\end{thmm}

Constructing global branches of the correspondence on the complex plane is, in general, obstructed by the presence of the algebraic singularity at zero. Remarkably, \textcolor{black}{after lifting $\mathbf{f}_c^{-1}$ to the unit disk,} such branches can indeed be defined, as the following theorem shows.

\begin{thmm}[\bf Contracting lifts]\label{ljksoijdfkwdfsed}
Suppose that $a$ is a Misiurewicz point of the family $\mathbf{f}_c$ given by~\eqref{lkjdllkjhalsoed}.
Let $\mathbf{g}$ be the holomorphic multifunction of Theorem~\ref{gdsadgwed}, defined on the open unit disk.
Every branch $f$ of $\mathbf{g}$ has a unique extension to a global branch $F \colon \mathbb{D} \to \mathbb{D}$.
Moreover, $F$ is strictly contracting with respect to the Poincar\'e metric.
Since $\varphi \circ F$ is a branch of $\mathbf{f}_a^{-1} \circ \varphi$, the following diagram commutes:

\begin{equation}\label{sdfsdfskkdfdwlkjlksded}
    \begin{tikzcd}
        \mathbb{D} \arrow{r}{F}\arrow{d}[swap]{\varphi} & \mathbb{D} \arrow{d}{\varphi} \\
        (\mathbb{C}, \nu) \arrow{r}[swap]{\mathbf{f}_a^{-1}} & (\mathbb{C}, \nu)
    \end{tikzcd}
\end{equation}

\end{thmm}

The following theorem will be restated later as Theorem~\ref{poiaudfposiuwedfs}.

\begin{thmm}[\bf Subhyperbolicity] \label{alkjasdpoipoallskdf}
Suppose that $a$ is a Misiurewicz point for the family~\eqref{lkjdllkjhalsoed}. Then $\mathbf{f}_a$ expands the orbifold metric~$\rho$ by a uniform factor in a neighborhood of~$K_a$. More precisely, there exists an open set~$V$ containing~$K_a$ and a constant~$\eta \in (0, 1)$ such that $R \cap V$ coincides with the unique bounded critical orbit of~$\mathbf{f}_a$, and for every univalent branch~$g$ of~$\mathbf{f}_a^{-1}$ defined on a region $W \subset V$, we have
\[
\|g'(w)\|_{\rho} < \eta \quad \text{for all } w \in W \setminus R.
\]
\end{thmm}

\begin{remark} \normalfont Notice that if $f$ is the local inverse of $g$ defined on $g(W)$, then \[\|f'(z)\|_{\rho}> \lambda\] for every $z$ in $g(W{\setminus} R)$, where $\lambda=1/\eta.$ 

\end{remark}

\gap \paragraph{Repelling cycles.} \label{gdasdfasdc} A forward orbit $(z_i)_0^{\infty}$ of the correspondence $\mathbf{f}_c$ is a \emph{cycle} of period $n$ if $z_i=z_{i+n}$  for every $i.$ Unless $z_i =0$  (the critical point) in which case $z_{i+1}=c$, there exists a unique univalent branch of $\mathbf{f}_a$ sending $z_i$ to $z_{i+1}.$ The composition of these (finitely many) branches along the cycle yields a univalent branch $f$ of $\mathbf{f}_a^n$ which has a fixed point at $z_0$. The  \emph{multiplier} is $\lambda = f'(z_0).$ Here the classical terminology applies, and the cycle is said to be \emph{repelling, geometrically attracting, indifferent, etc,} provided the same property is verified at the fixed point $z_0$ of $f$. A cycle that contains the critical point is called a \emph{critical cycle} or a \emph{super-attracting cycle.} 
  Notice that all known linearization results for analytic functions defined near a fixed point carry over naturally to correspondences through this association. 

Every point of a repelling cycle is a \emph{repelling periodic point.}

\begin{defi}[\bf Julia set]\label{adfadweerdccgqwe}\normalfont The closure of the set of repelling periodic points of $\mathbf{f}_c$ is the Julia set $J_c.$

\end{defi}

{\color{black}
\begin{remark}\label{asdfakjhsoidfsdf}\normalfont  By \cite[equation~(3)]{ETDS22}, the filled Julia set $K_c$ is given by the nested intersection
\[
K_c=\bigcap_{k\geq 0}\mathbf{f}_c^{-k}(D),
\]
where $D=\{\,|z|\le R\,\}$ is any sufficiently large disk. Thus $K_c$ is defined in complete analogy with the filled Julia set of a polynomial. An immediate consequence is that $K_c$ is compact and contains all cycles, including repelling ones. Hence the Julia set $J_c$, defined as the closure of the repelling cycles, satisfies $J_c\subset K_c$.

For polynomials one has $J_c=\partial K_c$. In contrast, except in the trivial case $c=0$ (where $J_0$ is the unit circle and $K_0$ the unit disk) in general $\partial K_c$ is a proper subset of $J_c$, with parts of $J_c$ lying in the interior of $K_c$. See \cite[Section~3]{BLS} for a more detailed discussion of the geometry of $K_c$ and $J_c$.

\end{remark}
}

The following theorem will later be reformulated as two separate statements: Theorem~\ref{lhlkjhapopoxooxdfs} and Theorem~\ref{fsdfawedsfasdeed}. As noted earlier, for a Misiurewicz point~$a$, the critical point is preperiodic to the cycle~$\alpha_a$, a fact we briefly recall here to clarify the next statement.

\begin{thmm}\label{lkjajhlkjhasdfoiqdfadf}
Suppose that $a$ is a Misiurewicz point for the family~\eqref{lkjdllkjhalsoed}. Then the critical point is preperiodic to a repelling cycle $\alpha_a$, which is contained in the filled Julia set $K_a$. Moreover, $J_a = K_a$, and
\[
K_a = \overline{\bigcup_{n \geq 0} \mathbf{f}_a^{-n}(0)}.
\]
\end{thmm}

Our proof that $J_a = K_a$ in this setting relies on the subhyperbolicity of $\mathbf{f}_a$ and the construction of homoclinic orbits. Alternative approaches, obtained by straightforward extensions of the standard arguments used to establish $J_a = K_a$ for polynomials, fail in the case of correspondences.

\gap
{\color{black}
\paragraph{\bf Outline.} 
The structure of this paper is as follows. Sections~\ref{asdfasdfbasdfcccds} and~\ref{adfasdflkhblkjhsoidf} present the  main results and related topics.  
Section~\ref{kjhkjhsdoiwe} discusses analytic continuation; while it contains no new results, it should not be skipped, as its primary purpose is to establish precise and consistent notation.
The proof of Theorem~\ref{gdsadgwed} is given in Section~\ref{asdfadsfqlkhdoihfqd}, and Theorem~\ref{ljksoijdfkwdfsed} is proved in Section~\ref{adfaoihpoiapoiasudfasdf}.  
Section~\ref{asdfasdfohasdfasd} establishes uniform contraction of branches of $\mathbf{f}_a^{-1}$ on compact sets, from which Theorem~\ref{alkjasdpoipoallskdf} follows; its proof appears in Section~\ref{asdfapoiuasdfasdrweg}.  
Section~\ref{asdfapoiuasdfasdrweg} also contains the proof of Theorem~\ref{lkjajhlkjhasdfoiqdfadf}, the final result of this paper.}

\section{The sheaf of germs of analytic functions}  \label{kjhkjhsdoiwe}

The content of this section summarizes key results from the classical exposition in L. Ahlfors \cite[Chapter~8]{Ahlfors}. While it is primarily based on that source, our presentation has the advantage of being more concise, allowing the reader to grasp the essential notation and ideas in just a few pages. Note that the notation used here differs slightly from that of Ahlfors.
In addition, we have introduced new definitions and related topics that do not appear in Ahlfors (or, to our knowledge, in any other standard exposition of the subject). For these results, we provide brief  proofs.

 \emph{Unless otherwise stated, every $U\subset \mathbb{C}$ is supposed to be a region, that is, a nonempty and connected open subset of the complex plane.}

\gap \paragraph{Germs.}
   For  a given $z\in \mathbb{C},$ let $S_z$ denote the space of all holomorphic maps locally defined in a neighborhood of $z$, with values in $\mathbb{C}.$ 
We define an equivalence relation on $S_z$  by setting $f\sim g$ provided $f$ and $g$ coincide on some small neighborhood of $z.$ Every element of $\mathfrak{S}_z=S_z/\mathord \sim$ is a \emph{germ of \textcolor{black}{an} analytic function at $z.$}  Let $U\subset \mathbb{C}$ be a region. Then   $\mathfrak{S}(U) = \cup_{z\in U} \mathfrak{S}_z$  is the \emph{sheaf of germs of analytic functions on $U,$} and the elements of $\mathfrak{S}_U$ are denoted by $\mathfrak{f},  \mathfrak{g}, \ldots$ If we need to specify that $\mathfrak{f}$ is a germ at $z$ represented by some holomorphic map $f,$ then we write $\mathfrak{f}=(f, z).$ Sometimes we denote $\mathfrak{S}(U)$ by $\mathfrak{S}_U.$ The sheaf $\mathfrak{S}_U$ is a topological space: if $\mathfrak{f}=(f, z)$, then a local basis at $\mathfrak{f}$ is determined by all sets of the form $\tilde{V} = \{ (f, \zeta) : \zeta \in V\}$, where $V$ is a region contained in the domain of $f$ and $z\in V.$ We define the continuous map $\pi: \mathfrak{S}_U \to U$ by $\pi(\mathfrak{f})=z$ if $\mathfrak{f}=(f,z).$

\gap \paragraph{Global analytic functions.} The connected components $\mathfrak{S}_{\alpha}(U)$ of $\mathfrak{S}(U)$ are Riemann surfaces. A local chart at $\mathfrak{f} \in \mathfrak{S}_{\alpha}(U)$ is simply a restriction of the projection $\pi.$ The restriction of $\pi$ to $\mathfrak{S}_{\alpha}(U)$ is denoted by $\pi_{\alpha}.$ The map $\tilde{\mathbf{f}}_{\alpha}: \mathfrak{S}_{\alpha}(U) \to \mathbb{C}$ which assigns every germ $(f,z)$ to $f(z)$ is evidently holomorphic on the Riemann surface $\mathfrak{S}_{\alpha}(U).$  By definition, each $\tilde{\mathbf{f}}_{\alpha}$ is a \emph{global analytic function}, and every germ of analytic function $\mathfrak{f}$ determines a unique global analytic function $\tilde{\mathbf{f}}_{\alpha}$ such that $\mathfrak{f} \in \mathfrak{S}_{\alpha}(U).$

\gap \paragraph{Analytic continuation.}\label{lkjslfjkdwdf} Given a continuous curve $\gamma: [a, b] \to U$ and a germ $\mathfrak{f}$ of a holomorphic function at $\gamma(a)$, an \emph{analytic continuation of $\mathfrak{f}$ along $\gamma$} is a continuous curve $\tilde{\gamma}: [a, b] \to \mathfrak{S}(U)$ that projects onto $\gamma,$ that is,  $\pi (\tilde{\gamma}(t)) = \gamma(t)$, for every $t.$ \textcolor{black}{It is not necessary to explicitly describe} the $\tilde{\mathbf{f}}_{\alpha}$ associated with the continuation $\tilde{\gamma}$, since $\{\tilde{\gamma}\}$ must be contained in a unique component $\mathfrak{S}_{\alpha}(U).$
 Two analytic continuations along the same curve are either identical or differ for every $t.$ If the analytic continuation exists, then the first germ $\mathfrak{g}_1=\tilde{\gamma}(a)$ uniquely determines the analytic continuation along $\gamma.$ Hence $\tilde{\gamma}$ is \emph{the analytic continuation of $\mathfrak{g}_1$} along $\gamma.$ In the same situation, we can say that $\mathfrak{g}_1$ and $\mathfrak{g}_2= \tilde{\gamma}(b)$ are \emph{joined by an analytic continuation} or that  the \emph{continuation of $\mathfrak{g}_1$ along $\gamma$ leads to the germ $\mathfrak{g}_2$ at $\gamma(b).$} For example, every pair of germs in a component $\mathfrak{S}_{\alpha}(U)$ can be joined by an analytic continuation $\tilde{\gamma}$.

\begin{defi}\label{jhlkjhsoieuhdsd} \normalfont We say that a continuous curve $\zeta$ in $\mathbb{C}$ defined on $[a, b]$  \emph{is given by an analytic continuation of a germ $\mathfrak{f}$ along another curve} $\gamma$ defined on $[a, b]$  if $$ \zeta(t) = f_t(\gamma(t)),$$ for every $t\in [a, b]$, where $f_t$ is a holomorphic map locally defined at $\gamma(t)$ such that \[\tilde{\gamma}(t) = (f_t, \gamma(t))\] is an analytic continuation of $\mathfrak{f}$ along $\gamma$ with $\tilde{\gamma}(a) = \mathfrak{f}.$ Equivalently, we may also say that the \emph{analytic continuation of $\mathfrak{f}$ along $\gamma$ yields $\zeta.$ }
\end{defi}

\gap \paragraph{Holomorphic multifunctions.}\label{lkoikdkdkciied}  By a \emph{holomorphic multifunction} $\mathbf{f}$ we mean a set-valued function that sends every point $z$ of $U$ to a subset $\mathbf{f}(z)$ of $\mathbb{C}$ in such a way that, whenever $w_0\in \mathbf{f}(z_0)$, there exists a holomorphic map $f: V \to \mathbb{C}$ defined on a region, called a \emph{branch} of $\mathbf{f}$, such that $f(z_0)=w_0$ and $f(z) \in \mathbf{f}(z),$ for every $z\in V.$  \textcolor{black}{If we wish to make the domain of $\mathbf{f}$ explicit,} we write $\mathbf{f}\colon U \to \mathbb{C}.$

(The usage of boldface letters avoids misinterpretations with single-valued holomorphic maps).

Any  global analytic function $\tilde{\mathbf{f}}_{\alpha}$  gives rise to a holomorphic multifunction \[\mathbf{f}_{\alpha}(z) = \tilde{\mathbf{f}}_{\alpha}(\pi_{\alpha}^{-1}(z))\] defined on $U.$ Given a holomorphic multifunction $\mathbf{g}: U \to \mathbb{C}$, the associated \emph{space of germs} $\mathfrak{S}(\mathbf{g})$ consists of all germs $(g, z)$ such that $g$ is a branch of $\mathbf{g}$ and $z$ is in the domain of $\mathbf{g}.$ It should be noticed that every branch is holomorphic and its domain is connected.

\begin{defi}[\bf Global analytic multifunction]\label{kjhlksbldfhkowd}\normalfont A holomorphic multifunction $\mathbf{g}: U \to \mathbb{C}$ is a \emph{global analytic multifunction} if $\mathfrak{S}(\mathbf{g})$ is a connected component of  $\mathfrak{S}_U.$ Hence $\mathfrak{S}(\mathbf{g})$ is the Riemann surface of $\mathbf{g}.$

\end{defi}

\begin{remark}\label{sdfojopwijdfdf} \normalfont  It is easy to show that every global analytic multifunction coincides with some $\mathbf{f}_{\alpha}.$  We do not know whether every $\mathbf{f}_{\alpha}$ is a global analytic multifunction. Indeed, the condition $\mathfrak{S}(\mathbf{f}_{\alpha}) = \mathfrak{S}_{\alpha}(U)$ is not a direct consequence of $\mathbf{f}_{\alpha} = \tilde{\mathbf{f}}_{\alpha}\circ \pi_{\alpha}^{-1}.$ 
\end{remark}

\begin{defi}[\bf Separable multifunction] \label{asdfasdfwdijpois} \normalfont A holomorphic multifunction $\mathbf{g}: U \to \mathbb{C}$ is separable if there exists a locally finite set of exceptional points $E=E({\mathbf{g}})\subset U$ \textcolor{black}{such that any point $x\in U {\setminus} E$} has a connected neighborhood $V \subset U {\setminus} E$ which is the domain of a family $\mathcal{F}$ of branches $f: V \to \mathbb{C}$, with pairwise disjoint images, such that any other local branch $h: W \to \mathbb{C}$ defined on a region $W \subset V$ is the restriction $h=f|_W,$ for some $f\in \mathcal{F}.$
\end{defi}

\begin{thm} If $\mathbf{f}_{\alpha}$ is separable, then $\mathbf{f}_{\alpha}$ is a global analytic multifunction. 

\end{thm}

\begin{proof} It suffices to show that $\mathfrak{S}(\mathbf{f}_{\alpha}) \subset \mathfrak{S}_{\alpha}(U)$. The other inclusion is immediate.  We will show that every $\mathfrak{g} \in \mathfrak{S}(\mathbf{f}_{\alpha})$ belongs to $\mathfrak{S}_{\alpha}(U).$ Without loss of generality, we may assume that  $\mathfrak{g}=(g, z_0)$, for some holomorphic branch $g: \Omega \to \mathbb{C}$  of $\mathbf{f}_{\alpha}$ defined on a region $\Omega \subset U$ such that $z_0$ is the only point of $E$ in $\Omega$   (see Definition \ref{asdfasdfwdijpois}). Choose $z_1 \neq z_0$ in $\Omega.$ Since $\mathbf{f}_{\alpha}$ is separable, there exists a connected neighborhood $V\subset \Omega{\setminus} \{ z_0\}$ of $z_1$ which is the domain of a family $\mathcal{F}$ of branches of $\mathbf{f}_{\alpha}$ with pairwise disjoint images satisfying the conditions of Definition \ref{asdfasdfwdijpois}. Hence $\mathbf{f}_{\alpha}(V)$ is given by the disjoint union of all $f(V)$ such that $f\in \mathcal{F}.$ Since $g(V)$ is connected, it must be contained in some $f_j(V)$, with $f_j \in \mathcal{F}.$ This implies $g(z) = f_j(z)$, for every $z\in V.$ Let $S_{g}$  be the subset of the sheaf given by all $(g, z)$ such that $z\in \Omega.$ Clearly, this set is connected in the sheaf topology, and we will show that it intersects $\mathfrak{S}_{\alpha}(U),$ which is enough to conclude that $S_{g}$ is contained in $\mathfrak{S}_{\alpha}(U).$ 
In order to prove this assertion, we first notice that $(f_j, z)$ is a germ in $\mathfrak{S}_{\alpha}(U)$, for every $z\in V.$ Indeed, using the definition of $\mathbf{f}_{\alpha}=\tilde{\mathbf{f}}_{\alpha}\circ \pi_{\alpha}^{-1}$ and the fact that $f_j$ is a branch of $\mathbf{f}_{\alpha}$, for any $z\in V$ there exists a holomorphic map $g_z$ locally defined at $z$ such that $(g_z, z) \in \mathfrak{S}_{\alpha}(U)$ and $f_j(z) = g_z(z).$ Let $W_z\subset V$ be a connected neighborhood of $z$ which is contained in the domain of $g_z.$ Again, it follows from the definition of $\mathbf{f}_{\alpha}$ that $g_z: W_z \to \mathbb{C}$ is a branch of $\mathbf{f}_{\alpha}.$  Since $\mathbf{f}_{\alpha}$ is separable, $g_z|_{W_z}=h_z|_{W_z}$, for some $h_z \in \mathcal{F}$. The maps in $\mathcal{F}$ have pairwise disjoint images over $V,$ and since $h_z(z) = f_j(z),$ it follows that $h_z=f_j$, for any $z\in V.$ Hence $$(f_j, z) = (h_z, z)=(g_z, z) \in \mathfrak{S}_{\alpha}(U), $$ for every $z\in V.$  Since $(g, z) = (f_j, z)$ belongs to the connected component $\mathfrak{S}_{\alpha}(U)$, it follows that $S_g$ intersects $\mathfrak{S}_{\alpha}(U)$, and since $S_{g}$ is connected, $S_g$ is contained in $\mathfrak{S}_{\alpha}(U)$.  We conclude that $(g, z_0)$ belongs to $S_g \subset \mathfrak{S}_{\alpha}(U).$ Hence $\mathfrak{g} 
\in \mathfrak{S}_{\alpha}(U)$, as desired. 
\end{proof}

%%%%%%%%

\gap \paragraph{Examples.} The elementary multifunction $\log(z)$ is a global analytic multifunction defined on $\mathbb{C}^*.$ Every irreducible complex polynomial in two complex variables $P(z, w)$ determines a global analytic multifunction $\mathbf{g}(z) =\{w : P(w, z)=0\}.$  In this case we have to remove from $\mathbb{C}$ a locally finite set of (algebraic) singularities (see page \pageref{khskljhdfw}). Any $\mathbf{g}$ arising from some $P(z, w)$ in this way is an \emph{algebraic function.}\footnote{More appropriately, an \emph{algebraic multifunction}. However, the term algebraic function is standard, see \cite[p. 300]{Ahlfors}}   If $P(z, w)$ is not irreducible then $\mathbf{g}$ is no longer a global analytic multifunction. Nevertheless, $\mathfrak{S}(\mathbf{g})$ is the union of finitely many connected components $\Lambda_i$ of the sheaf $\mathfrak{S}_{U}$, where $U=\mathbb{C} {\setminus}B$ and $B$ is the set of singularities.  This is the case of the correspondence $(w-c)^2 =z^4,$ which is completely determined by two global branches: $w=z^2 +c$ and $w=-z^2 +c.$ The singularity at $0$ is removable in this case.

\noindent Recall from Definition \ref{asdfasdfwdijpois} that $E=E(\mathbf{g})$ is the set of exceptional points of a separable holomorphic multifunction $\mathbf{g}.$

\begin{thm}\label{jdjfsfdwef} 
Suppose that $\mathbf{g}: U \to \mathbb{C}$ is a separable holomorphic multifunction with $E(\mathbf{g})=\emptyset$ and let $(\gamma, \eta)$ be a pair of continuous curves defined on $[a, b]$, with $\{ \gamma\} \subset U$ and $\eta(t) \in \mathbf{g}(\gamma(t))$, for every $t\in [a, b].$ Then for every $t$ there exists a unique  germ \[\mathfrak{f}_t=(f_t, \gamma(t))\] determined by a local branch $f_t$ of $\mathbf{g}$ at $\gamma(t)$ which sends $\gamma(t)$ to $\eta(t).$ Moreover, the curve $\tilde{\gamma}(t)= \mathfrak{f}_t$ is continuous on $[a,b]$. Hence $\eta$ is given by the analytic continuation of $\mathfrak{f}_a$ along $\gamma.$ 
  \end{thm}

\begin{proof}  Since $\{\gamma\}$ is compact and $E=\emptyset$,  the following property holds for some $\epsilon >0$: for any $t \in [a,b],$ there exists a unique local branch $f_t$ defined on the open ball $D_t$ centered at $\gamma(t)$, with radius $\epsilon$, such that $f_t$ sends $\gamma(t)$ to $\eta(t)$. There exists $\delta >0$ such that $D_t \cap D_s$ is  a region and $f_t = f_s$ on $D_t \cap D_s$ if $|t-s| < \delta$ on $[a, b].$ Hence $\tilde{\gamma}(t) = (f_t, \gamma(t))$ is a continuous path $[a, b] \to \mathfrak{S}(U),$ which is evidently an analytic continuation satisfying the conditions of the statement. 
\end{proof}

\begin{thm}[\bf Monodromy] \label{gjdsdgwe}   Let  $U \subset \mathbb{C}$ be a region. Assume that $\mathfrak{S}_{\alpha}(U)$ is a component of the sheaf  such that any germ $\mathfrak{f}$ in $\mathfrak{S}_{\alpha}(U)$ can be continued along any curve  in $U$ starting at $\pi(\mathfrak{f}).$ Let $\gamma_1$  and $\gamma_2$ be homotopic curves in $U$, defined on $[a, b]$, with $\gamma_1(a)=\gamma_2(a)$ and $\gamma_1(b)=\gamma_2(b).$   Then the continuations of a given $\mathfrak{f} \in \mathfrak{S}_{\alpha}(U)$ with $\pi(\mathfrak{f})=\gamma_1(a)$  along $\gamma_1$ and $\gamma_2$ lead to the same germ at $\gamma_1(b).$ \end{thm}

\begin{proof} See \cite[p. 295]{Ahlfors}.
\end{proof}

\begin{cor}\label{kjhlkjhlskjed} Under the same hypothesis of Theorem \ref{gjdsdgwe}, if  $U$ is simply connected, then for  any germ $\mathfrak{f} \in \mathfrak{S}_{\alpha}(U)$ there exists a branch $f: U \to \mathbb{C}$ such that  $\mathfrak{f}=(f, \pi(\mathfrak{f})).$ Equivalently, every local branch can be extended  to a branch defined on the whole space. 
\end{cor}

\begin{proof} Given $z\in U$, we choose some curve $\gamma$ in $U$ connecting $\pi(\mathfrak{f})$ to $z.$ The continuation of $\mathfrak{f}$ along $\gamma$ leads to a germ $\mathfrak{g}_z=(g_z,z)$. By the Monodromy Theorem, $\mathfrak{g}_z$ does not depend on the choice of $\gamma,$ and we may define $f(z) = g_z(z).$ Since a small perturbation of $z$ produces essentially the same map $g_z$, we have a global branch $f:  U \to \mathbb{C}$ with $\mathfrak{f}=(f,\pi(\mathfrak{f})).$
\end{proof}

\gap \paragraph{Singularities.}\label{khskljhdfw}
If a global analytic multifunction is defined on a region $U$ except for a locally finite set of points $x_j$ in $U$,   then we say that each $x_j$ is a \emph{singularity} of $\mathbf{g}.$ If we can extend $\mathbf{g}$ to a global analytic multifunction on $U$, then by definition each $x_j$ is a \emph{removable singularity}.

  \begin{thm}[\bf Removable singularities]\label{lkhlkjhsiedf} Let $\mathbf{h}$ be a global analytic multifunction defined on a region $U$ except for a locally finite set of singularities $x_j$ in $U.$ Let $A_{r,j}$ be the punctured neighborhood of $x_j$ consisting of all $z$  with $0< |z-x_j | < r.$ Let $\gamma$ be a small circle $x_j+ re^{i2\pi t}$  around the  singularity $x_j$ with parameter $t$ in $[0, 1].$  
  
  If the continuation of every germ of $\mathbf{h}$ at $x_j +r$ along $\gamma$ leads back to itself, \textcolor{black}{then for every germ $\mathfrak{h} \in \mathfrak{S}(\mathbf{h})$} with $\pi(\mathfrak{h}) \in A_{r, j}$ there exists a unique branch $F$ of $\mathbf{h}$ defined on $A_{r,j}$  such that $F$ defines the germ $\mathfrak{h}$ at $\pi(\mathfrak{h})$, that is, $(F, \pi(\mathfrak{h}))=\mathfrak{h}.$ In other words, local branches within $A_{r,j}$ can be extended to a punctured neighborhood of the singularity.
  Let $\mathcal{B}_j$ denote the space of all branches $F$ defined on $A_{r, j}$ obtained in this way. 
  
  If every $F\in \mathcal{B}_j$ has a removable singularity at $x_j$, for every $j$, then $\mathbf{h}$ has a unique extension to a global analytic multifunction  $\mathbf{g}$ defined on $U$ which coincides with $\mathbf{h}$ at nonsingular points, and $\mathbf{g}(x_j)$ is given by $\{F(x_j): F\in \mathcal{B}_j\}$ at every singular point $x_j.$
  \end{thm}
  
  \begin{proof}[Reference to the proof] See \cite[p. 297-300]{Ahlfors}. On page 299 of \cite{Ahlfors} Ahlfors describes \emph{ordinary algebraic singularities}. Each $x_j$ in Theorem \ref{lkhlkjhsiedf} is an ordinary algebraic singularity;  the function $F\in \mathcal{B}_j$ is obtained by analytic continuation on page 298 by $F(\zeta)=f(\zeta^h),$  and  is the same  Laurent development   with $h=1$ which is present on page 299 of \cite{Ahlfors}. 
  \end{proof}

\section{Proof of Theorem \ref{gdsadgwed} }\label{asdfadsfqlkhdoihfqd}
We will divide the proof into small steps.

\begin{lem}\label{alkjacbcbkhlpsdf}
The set of ramified points defined in Definition~\ref{lkjlhdoiehfeddd} is locally finite.
\end{lem}

\begin{proof} Apply \cite[Lemmas 2.1 and 2.2]{ETDS22}.  Since $a$ is Misiurewicz, we have only one bounded forward orbit of $0$ under $\mathbf{f}_a,$ which is strictly preperiodic. There exist at most finitely many points of the postcritical set within any open ball $B(0, r).$ 
\end{proof}

Using the same notation of the statement of Theorem \ref{gdsadgwed}, 
we know that the set of branch points $B=\{z \in \mathbb{D}: \varphi'(z) =0\}$ is locally finite and  its complement $U=\mathbb{D}{\setminus} B$ is a region.

 \begin{lem}\label{$(a)$}The restriction $\mathbf{h}= \mathbf{g}|_{U}$ is a separable holomorphic multifunction with $E(\mathbf{h})=\emptyset,$ where $E(\mathbf{h})$ is as in  Definition \ref{asdfasdfwdijpois}.
 
 \end{lem}

 \begin{proof}
 We know that  $\varphi: \mathbb{D} {\setminus} B \to \mathbb{C} {\setminus} R$ is a covering map, where $R=\varphi(B)$ is the set of ramified points. If $V$ is a sufficiently small  open disk in $\mathbb{D}{\setminus} B$ then $\varphi(V)$ is also a small conformal disk in $\mathbb{C} {\setminus} R$ which is mapped under $\mathbf{f}_a^{-1}$ to $p$ disjoint conformal disks  $D_j$ in $\mathbb{C}{\setminus} R,$ each of which   evenly is covered by $\varphi,$ that is, $\varphi$ maps every connected component of $\varphi^{-1}(D_j)$ onto $D_j$ as a univalent map. Thus $\mathbf{h}(V)$ consists of disjoint conformal disks, each of which being the image of a univalent branch of $\mathbf{h}.$ The  conditions that define a separable holomorphic multifunction (def. \ref{asdfasdfwdijpois}) are clearly implied by this property.
 \end{proof}

\begin{lem} \label{$(b)$} We consider the partition of $\mathfrak{S}(\mathbf{h})$ into connected components $\Lambda_{\alpha}.$ Each $\Lambda_{\alpha}$ is a connected component of the sheaf $\mathfrak{S}_{U}.$

\end{lem}

 \begin{proof} It suffices to show that every germ $\mathfrak{h}$ in the connected component $\tilde{\Lambda}_{\alpha}$  of $\mathfrak{S}_{U}$  that contains $ \Lambda_{\alpha}$ is determined by a local branch of $\mathbf{h}$, for then $\Lambda_{\alpha} = \tilde{\Lambda}_{\alpha}$ follows immediately. Since $\tilde{\Lambda}_{\alpha}$ is path connected and contains $\Lambda_{\alpha},$ we join $\mathfrak{h}$ to a germ $\mathfrak{g} \in \Lambda_{\alpha}$ using a continuous curve $\tilde{\gamma}:[0,1] \to \tilde{\Lambda}_{\alpha}$  such that $\tilde{\gamma}(0)=\mathfrak{g}$ and $\tilde{\gamma}(1) = \mathfrak{h}.$ Clearly $\tilde{\gamma}$ is an analytic continuation of $\mathfrak{g}$ along $\gamma:=\pi \circ \tilde{\gamma}.$ Using a standard compactness argument we find finitely many holomorphic maps $f_j: \Omega_j \to \mathbb{D}$ defined on open conformal disks $\Omega_j$ such that $(\Omega_j)_1^n$ is a covering of $\gamma([0,1])$ and  $\Omega_j \cap \Omega_{j+1} \neq \emptyset$ for every $j.$ We may also assume that for some partition
 \[0=t_1 < t_2 < \cdots < t_{n+1}=1 \]
  \noindent we have $\gamma([t_j, t_{j+1}])  \subset \Omega_j$ and $\tilde{\gamma}(t) = (f_j, \gamma(t))$ as  $t\in [t_j, t_{j+1}]$, for every $j.$ Since $\mathfrak{g}\in \Lambda_{\alpha}$ is a germ determined by $f_1$ at $\gamma(0),$ it follows that $f_1(z) \in \varphi^{-1} \circ \mathbf{f}_a^{-1} \circ \varphi(z)$, whenever $z\in \Omega_1.$ This means that  \begin{equation} \label{gkdlskejref} (\varphi \circ f_j(z))^p = (\varphi(z) -a)^{q} \end{equation} if $j=1$ and $z\in \Omega_1.$ Since $f_1$ coincides with $f_2$ on the intersection of their domains, \eqref{gkdlskejref} holds for $j=2$ and $z\in \Omega_1 \cap \Omega_2.$ By the Identity  Theorem, the corresponding equation \eqref{gkdlskejref} for $j=2$ actually holds for any $z$ in $\Omega_2.$ We may repeat this argument inductively. The conclusion: \eqref{gkdlskejref} holds for $z\in \Omega_j$, for every $j.$ In particular, for $j=n$ this implies $f_n(z) \in \varphi^{-1} \circ \mathbf{f}_a^{-1} \circ \varphi(z)$, for every $z\in \Omega_n.$  Hence $\mathfrak{h}=(f_n, \gamma(1))$ is determined by a local branch of $\mathbf{h}.$
  \end{proof}

\begin{lem} \label{$(c)$} If $\pi$ is the standard projection of $\mathfrak{S}_U$ onto $U,$ then $\pi(\Lambda_{\alpha}) = U.$

\end{lem}

   \begin{proof} Fix a point $z_0$ in $U.$ Any $z\in U$ is connected to $z_0$ by a curve $\gamma$ defined on $[0,1]$ with $\gamma(0) = z_0$ and $\gamma(1) =z.$ First, suppose that   the analytic continuation $\tilde{\gamma}$ of any germ $\mathfrak{h} \in \Lambda_{\alpha}$ with $\pi(\mathfrak{h}) = z_0$ along $\gamma$ always exists. Since $z$ is arbitrary and  $\Lambda_{\alpha}$ is connected, it follows that $\tilde{\gamma}(1)$ is a germ in $\Lambda_{\alpha}$  whose projection is $z$, thereby proving that $\pi(\Lambda_{\alpha}) = U.$ Now we will check that  the continuation of any germ $\mathfrak{h}= (h, z_0)$ along $\gamma$ exists. The local branch $h$  of $\mathbf{h}$ is defined on a neighborhood of $z_0$ and is given by the composition $g\circ f \circ \varphi,$ where $f$ is a local branch of $\mathbf{f}_{a}^{-1}$ at $\varphi(z_0)$ and $g$ is a local inverse of $\varphi$ at $f(\varphi(z_0)).$ Since $\gamma$ does not intersect $B$, its image $\gamma_1=\varphi \circ \gamma$ is  contained in $\mathbb{C} {\setminus} R.$ Considering the global analytic  multifunction $\mathbf{f}_a^{-1}: \mathbb{C}{\setminus} \{a\} \to \mathbb{C}{\setminus} \{0\}$, the continuation of $(f, \varphi(z_0))$ along  $\gamma_1=\varphi \circ \gamma$ (see def. \ref{jhlkjhsoieuhdsd}) yields a curve $\gamma_2$ in $\mathbb{C}{\setminus} R.$  As a consequence of the Lifting Theorem for covering spaces,  the inverse $\varphi^{-1}: \mathbb{C}{\setminus} R \to \mathbb{D}{\setminus} B$ of the covering map is a global analytic multifunction and $(g, f(\varphi(z_0)))$ is a germ of $\varphi^{-1}$; the continuation of this germ along $\gamma_2$ yields a curve  $\gamma_3.$ Notice that $\gamma_3(t)$ is in $\mathbf{h}(\gamma(t))$, for every $t.$ Taking into account  Lemma \ref{$(a)$} and Theorem \ref{jdjfsfdwef},
the final curve $\gamma_3$ is given by an analytic continuation of $\mathfrak{h} \in \mathfrak{S}(\mathbf{h})$ along $\gamma, $ as desired.  
\end{proof}

\begin{lem} \label{$(d)$}
There exists a unique global analytic multifunction $\mathbf{h}_{\alpha}: U \to \mathbb{D}$ such that $\mathfrak{S}(\mathbf{h}_{\alpha}) = \Lambda_{\alpha}.$  Any germ $\mathfrak{h} \in \mathfrak{S}(\mathbf{h}_{\alpha})$ can be continued along any curve $\gamma:[a, b] \to U$ starting at $\pi(\mathfrak{h}).$ For every $z\in U$, we have \begin{equation}\label{lkjkdjewed} \mathbf{h}(z)= \bigcup_{\alpha} \mathbf{h}_{\alpha}(z).\end{equation} 
\end{lem}

\begin{proof} We must define $\mathbf{h}_{\alpha}$ explicitly. By definition,  if $z\in U, $ then $w$ belongs to $\mathbf{h}_{\alpha}(z)$ whenever $w=h(z)$, for some germ $\mathfrak{h}=(h,z)$ in  $\Lambda_{\alpha}$ with $\pi(\mathfrak{h})=z.$ In this case,  it should be noticed that $(h, \zeta)$ is in $\Lambda_{\alpha}$ for every $\zeta$ in the domain of $h,$ for then $\Lambda_{\alpha}$ is a connected component of the sheaf $\mathfrak{S}_{U}.$ This shows that every $h$ used to define a value of $\mathbf{h}_{\alpha}(z)$ can also be used to define $h(\zeta)$ as a value of $\mathbf{h}_{\alpha}(\zeta),$ for every $\zeta$ in the domain of $h.$ Hence $h$ is a branch of the multifunction $\mathbf{h}_{\alpha}.$ It follows  that $\mathbf{h}_{\alpha}$ is a holomorphic multifunction. 
 
 We will show that   $\mathfrak{S}(\mathbf{h}_{\alpha})= \Lambda_{\alpha}.$  For the first inclusion $\Lambda_{\alpha} \subset \mathfrak{S}(\mathbf{h}_{\alpha})$, let $\mathfrak{h}$ be a germ in $\Lambda_{\alpha},$ which is a component of $\mathfrak{S}(\mathbf{h}).$  Then $\mathfrak{h}=(h, z_0)$ for some holomorphic branch of $\mathbf{h}$ and some $z_0$ in the domain of $h.$ By Lemma 
\ref{$(b)$}, $(h, z)$ is a germ in $\Lambda_{\alpha}$, for every $z$ in the domain of $h.$ According to the definition of $\mathbf{h}_{\alpha}$, $h(z)$ belongs to $\mathbf{h}_{\alpha}(z),$ for every $z$ in the domain of $h.$  Hence $h$ is a branch of $\mathbf{h}_{\alpha}$ and $\mathfrak{h} \in \mathfrak{S}(\mathbf{h}_{\alpha})$, which shows that $\Lambda_{\alpha}$ is contained in $\mathfrak{S}(\mathbf{h}_{\alpha}).$ 
 
 For the converse inclusion $\mathfrak{S}(\mathbf{h}_{\alpha}) \subset \Lambda_{\alpha}$, assume $\mathfrak{h}$ is  a germ in $\mathfrak{S}(\mathbf{h}_{\alpha}).$  Then $\mathfrak{h}=(h, z_0)$, for some branch $h$ of $\mathbf{h}$ and some $z_0$ in the domain of $h.$ Using the definition of $\mathbf{h}_{\alpha}$, for every $z$ in the domain of $h$ we find a holomorphic map $g_z$ locally defined at $z$ such that $(g_z, z)$ is a germ in $\Lambda_{\alpha}$ and $h(z)=g_z(z).$ By Lemma \ref{$(b)$}, $\Lambda_{\alpha}$ is a connected component of the sheaf, and $$(g_z, \zeta) \in \Lambda_{\alpha} \subset \mathfrak{S}(\mathbf{h}), $$ for every $\zeta$ in the domain of $g_z.$ This means that $g_z$ is a branch of $\mathbf{h}$ when restricted to some small neighborhood of $\zeta$, for every $\zeta$ in the domain of $g_z.$ We conclude that $g_z$ is a branch of $\mathbf{h}.$  Since $h(z)=g_z(z)$, $h$ is also a branch of $\mathbf{h}.$ \textcolor{black}{By Lemma  \ref{$(a)$},}  $\mathbf{h}$ is separable with $E(\mathbf{h})=\emptyset.$  Both $g_{z_0}$ and $h$ are branches of $\mathbf{h}$ locally defined at $z_0.$ According to  Definition \ref{asdfasdfwdijpois}, for some neighborhood $V$ of $z_0$, either $h(V)$ and $g_{z_0}(V)$ are disjoint, or else $h=g_{z_0}$ when restricted to $V.$ But $h(z_0)=g_{z_0}(z_0)$. Therefore, $$\mathfrak{h} = (h, z_0) = (g_{z_0},  z_0) \in \Lambda_{\alpha},$$ which shows that $\mathfrak{S}(\mathbf{h}_{\alpha})$ is contained in $\Lambda_{\alpha}$, as desired.

  From Definition \ref{kjhlksbldfhkowd}, Lemma  \ref{$(b)$} and   $\mathfrak{S}(\mathbf{h}_{\alpha})=\Lambda_{\alpha}$ we conclude that   $\mathbf{h}_{\alpha}$ is a global analytic multifunction defined on $U.$ Now the sets $\mathfrak{S}(\mathbf{h}_{\alpha})$ provide a partition of $\mathfrak{S}(\mathbf{h})$, from  which \eqref{lkjkdjewed} follows immediately. The uniqueness of $\mathbf{h}_{\alpha}$ is implicit from its construction, that is, any other $\tilde{\mathbf{h}}$ with $\mathfrak{S}(\tilde{\mathbf{h}})=\Lambda_{\alpha}$ is necessarily given by the same explicit formulation of $\mathbf{h}_{\alpha}.$

  For the existence of the continuation of $\mathfrak{h}=(h, z_0)$ along $\gamma$ we first notice that $\gamma$ projects onto $\varphi\circ \gamma.$  Since $\mathbf{f}_a^{-1}$ is an algebraic multifunction on $\mathbb{C}{\setminus} \{a\}$, any germ of $\mathbf{f}_a^{-1}$ at the first point of $\varphi\circ \gamma$ can be continued along this curve. The continuation yields a curve $\zeta$ in the plane avoiding the set of ramified points $R$ such that $\zeta(t) \in \mathbf{f}_{a}^{-1}(\varphi \circ \gamma(t))$, for every $t,$ with $\zeta(0)=\zeta_0,$ where  $\zeta_0=\varphi(w_0)$  and $w_0=h(z_0).$   Since $\varphi: \mathbb{D}{\setminus} B \to \mathbb{C} {\setminus} R$ is  a covering map, {\color{black} by the lifting theorem, there exists a unique curve $\eta$ in the disk with $\eta(0)=w_0$ such that $\varphi \circ \eta = \zeta$.}  Now $\eta(t) \in \mathbf{h}_{\alpha}(\gamma(t))$ for every $t$, so that by Theorem \ref{jdjfsfdwef} $\eta$ is given by an analytic continuation $\tilde{\gamma}$ of $\mathfrak{h}$ along $\gamma.$ In particular, the continuation exists. 
  \end{proof}

\begin{lem} \label{$(e)$}Every point of $B$ is a removable singularity of $\mathbf{h}_{\alpha},$ which has a unique extension to a global analytic multifunction   $\mathbf{g}_{\alpha}: \mathbb{D} \to \mathbb{D}$ such that  $\mathbf{g}_{\alpha}(z)$ is contained in  $\mathbf{g}(z)$, for every $z$ in $\mathbb{D}.$

\end{lem}

 \begin{proof} The set of branch points $B \subset \mathbb{D}$ is  defined in the first line of the proof of the theorem. Recall that $\varphi(B)= R$ is the set of ramified points.  Since the projection $\varphi$ is regular, $\varphi^{-1}(R)=B$.
 
 Let $x\in \mathbb{D}$ be a branch point.  We will show that the continuation of any germ $\mathfrak{h}=(h, x+r)$ of  $\mathbf{h}_{\alpha}$ along a small circle $\gamma(t) = x +re^{i2\pi t}$  ($t\in [0,1]$) leads back to the same germ $\mathfrak{h}$ at $\gamma(1)$, which is enough to conclude that every branch point $x$ is a removable singularity of $\mathbf{h}_{\alpha}$,  according to  Theorem \ref{lkhlkjhsiedf}.

Every image $y\in \mathbf{h}_{\alpha}(x)$ of the branch point $x$ comes from a sequence $x\mapsto a_j \mapsto b \mapsto y$ where $\varphi$ sends $x$ to a ramified point $a_j$, $b\in \mathbf{f}_a^{-1}(a_j)$ and $y\in \varphi^{-1}(b):$ 
\begin{equation}\label{sdfsdfskkdfdw}
    \begin{tikzcd}
        (\mathbb{D}, x) \arrow{r}{\mathbf{h}_{\alpha}}\arrow{d}[swap]{\varphi} & (\mathbb{D}, y) \arrow{d}{\varphi} \\
        (\mathbb{C}, a_j) \arrow{r}[swap]{\mathbf{f}_a^{-1}} & (\mathbb{C}, b)
    \end{tikzcd}
\end{equation}

  In Lemma \ref{$(e1)$} we will make the choice of $b$ and $y$ unique, depending only on $x$ and $\mathfrak{h}=(h,x+r).$  

\textcolor{black}{By  Lemma \ref{$(d)$},} if $\zeta:[0, 1] \to U\cup \{x\}$ is a curve starting at $x+r$ and terminating at $x,$ then the analytic continuation of $\mathfrak{h}$ along the restriction of $\zeta$ to any closed subinterval $[0, t] \subset [0, 1)$ exists. If we take $t\to 1$ we conclude that there exists an extension of the domain of $h$ to an open set $V\subset U$ with a sequence $x_k\in V$ converging to $x$, such that $V$ has diameter $r+\epsilon$, for $\epsilon>0$ arbitrarily small. (Alternatively, we may apply Corollary \ref{kjhlkjhlskjed} to reach the same conclusion, requiring $V$ to be simply connected).

\begin{lem} \label{$(e1)$}
Under the above conditions on $V$, the limit of $h(x_k)$ as $k\to \infty$ exists and is independent of the particular choice of $V$ and  $\{x_k\}.$   We denote this limit by $y$ and let $b=\varphi(y).$
\end{lem}

\begin{proof}
The local branch $h$ is given by a composition $\psi \circ f^{-1} \circ \varphi$, where $f^{-1}$ is a univalent branch of $\mathbf{f}_a^{-1}$, $\psi$ is a univalent branch of $\varphi^{-1}$. Since $\varphi$ is continuous, $\varphi(x_k)$ converges to $a_j =\varphi(x).$ The sequence $f^{-1} \circ \varphi (x_k)$ also converges to a point $b$ which is independent of $\{x_k\}$ and $V$: if $a=a_j$ then take $b=0$; otherwise, $f^{-1}$ can be extended to a domain that includes $a_j$, for then $a_j$ is in the closure of the domain of $f^{-1}$. In this case, $b=f^{-1}(a_j).$

Since $f^{-1}$ is branch of $\mathbf{f}_a^{-1}$, which is a continuous set-valued function in the Hausdorff topology, if the diameter of $V$ is small, then $f^{-1} \circ \varphi(V)$ is also small.  Hence we may assume that $f^{-1}\circ \varphi(V)$ is contained in a connected $W$, with $b\in W$, such that $W$ is evenly covered by $\varphi$, that is, each connected component of $\varphi^{-1}(W)$ projects onto $W$ by means of a proper map. By applying $\psi$, which is a local inverse of $\varphi$, we conclude that $h(V)$ is contained in one of these components, say $C_0.$ If $\nu(b)=1$, by reducing $V$ (and consequently $W$), we may assume that $\varphi: C_0 \to W$ is bi-holomorphic. It is clear that $h(x_k)$ converges to $(\varphi|_{C_0})^{-1}(b)$, which is independent of $V$ and $\{x_k\}.$ 

If $\nu(b)\neq 1,$ then every point of $\varphi^{-1}(b)$ is a branch point. Since $B$ is locally finite and the group of deck transformations $\Gamma$ acts transitively on the components of $\varphi^{-1}(W)$, with $\Gamma. B=B,$ by reducing $V$  if necessary, we may assume that every component of $\varphi^{-1}(W)$ has only one branch point. Let $y$ be the only branch point in $C_0$ (notice that $y$ does not depend on $V$ and $\{x_k\}$). As a consequence of B\"ottcher's Theorem, $\varphi$ is locally conjugate to $z^{\nu(b)}$ (up to translation) near $y$, and $\varphi(y)=b.$  Hence the image of every point of the sequence $b_k=f^{-1} \circ \varphi(x_k)$ under $(\varphi|_{C_0})^{-1}$ is a set of $\nu(b)$ points converging to $\{y\}$ in the Hausdorff topology. Thus, 
\[h(x_k) \subset (\varphi|_{C_0})^{-1}(b_k) \to \{y\}. \]

\noindent Hence $h(x_k) \to y$ in every case. The proof of Lemma \ref{$(e1)$} is complete.  \end{proof}

\noindent{\it Continuation of the proof of  Lemma \ref{$(e)$}}. We are ready to show that each branch point $x$ is a removable singularity of $\mathbf{h}_{\alpha}.$ Combining all values of $(\nu(a_j), \nu(b))$ yields $(p, p)$, $(p, q)$, $(p, 1)$, $(q, p)$, $(q, q)$ and $(q, 1).$ As we shall see, not every combination is possible.  In all possible cases, we will prove that the continuation of $\mathfrak{h}=(h, x+r)$ along the small circle $\gamma(t) = x +re^{i2\pi t}$ leads back to $\mathfrak{h}.$

Case $(p, p)$ means that $\nu(a_j)=p$ and $\nu(b)=p$. By Definition \ref{lkjlsjodijpoiefd}, this  implies $a_j\neq 0, a_j\neq a$ and $b=a_k\neq 0.$  Since the local degree of $\varphi$ at $x$ is $\nu(a_j)=p$, the image curve $\tilde{\gamma} = \varphi \circ \gamma$ has winding number

$$n(a_j, \tilde{\gamma}) = \frac{1}{2\pi i} \int_{\tilde{\gamma}} \frac{1}{z-a_j} dz =p.$$

\noindent Since $r$ is small and  $a_j\neq a,$  the univalent map $f^{-1}$ can be extended to  a connected neighborhood of $a_j$ that includes $\{ \tilde{\gamma}\}.$ The image curve $\tilde{\eta}=f^{-1}\circ \tilde{\gamma}$ also has winding number $n(\tilde{\eta}, b)$ equal to $p.$ Now $b=a_k$ is a ramified point with index $\nu(a_k)=p$, and the local degree of $\varphi|_{C_0}$ at $y$  is precisely $p,$ for then $\varphi|_{C_0}$ is conjugate to $z^p$, up to a translation, as described in the \textcolor{black}{proof of Lemma \ref{$(e1)$}.} (The following term \emph{yields} has a rigorous meaning, see  Definition \ref{jhlkjhsoieuhdsd}). Hence the analytic continuation of the germ $$(\psi, f^{-1} \circ \varphi(x+r)) $$ along  $\tilde{\eta}$ yields a curve $\eta$ in $C_0$ around $y$ which has winding number equal to $n(\eta, y)=1.$ 
Therefore, in the case $(\nu(a_j), \nu(b)) = (p, p)$ we have obtained a sequence of curves $(\gamma, \tilde{\gamma},  \tilde{\eta},   \eta)$  with a respective finite sequence of winding numbers
 \begin{equation}\label{lkjlksjdfwefsds} (n(\gamma, x),  n(\tilde{\gamma}, a_j), n(\tilde{\eta}, b), n(\eta, y)) \end{equation} 
 
 \noindent which equals  $(1, p, p, 1).$
\noindent The case $(\nu(a_j), \nu(b)) = (q, q)$ is impossible, for then $a_j =0$ and $b=0,$ and the critical point could not be strictly preperiodic (Misiurewicz). Case $(q, p)$ is also impossible, by a similar reasoning. In all other cases, we can  obtain different sequences of curves and winding numbers, which we indicate using the same notation from \eqref{lkjlksjdfwefsds}.   In case $(p, q)$, for example, we only have to take into account that in the equation $(w-c)^q = z^p$  which defines $\mathbf{f}_a$, the continuation of a curve $\zeta$ with winding number $n(\zeta, 0)=q$ in the $z$-plane yields a curve $\tilde{\zeta}$ with winding number $n(\tilde{\zeta}, a)=p$ in the $w$-plane.  We summarize as follows: 

 \begin{enumerate} \item Case $(p , p)$: sequence of winding numbers \eqref{lkjlksjdfwefsds} given by  $(1, p, p, 1).$
   \item Case $(p, q):$ sequence of winding numbers $(1, p, q, 1).$
 \item Case $(p, 1)$: sequence  of winding numbers $(1, p, p, p).$
 \item Case $(q, p)$: impossible (because $a$ is Misiurewicz).
 \item Case $(q, q)$: idem. 
 \item Case $(q, 1)$: sequence of winding numbers $(1, q, q, q).$
 
 \end{enumerate}

 As we shall see, what is relevant for us  is that the terminating curve $\eta$ is closed in all situations (this is a consequence of the sequence of winding numbers).  By Lemma \ref{$(e1)$} and its proof, it is implicit that $(h, x+r)$ is defined for every $r>0$ sufficiently small, for then the domain of $h$ can be extended by analytic continuation to an open set $V$ containing every interval $(0, r]$, for $r>0$ sufficiently small. The dependence of $\gamma$ and $\eta$ on $(h, x+r)$ can be made explicit by writing $\gamma_r$ and $\eta_r.$ While $h$ is the same map, we may take $r\to 0$ and check that $\eta_r$ converges uniformly to $y.$ Moreover,   $\eta_{r}(t)\in \mathbf{h} (\gamma_r(t))$, for every $t.$ Since $\mathbf{h}$ is a separable multifunction and $\eta_r(0)= h(\gamma_r(0)),$ we may apply Theorem \ref{jdjfsfdwef} to conclude that $\eta_r$ is given by the analytic continuation of the germ $\mathfrak{h}=(h, x+r)$ of $\mathbf{h}_{\alpha}$ along $\gamma_r$. Since the winding number is an integer, the curve $\eta_r$ must be closed, and  the analytic continuation of $(h, x+r)$ along $\gamma_r$ leads back to itself. We conclude from Theorem \ref{lkhlkjhsiedf} that each $\mathbf{h}_{\alpha}$ has a unique extension to a global analytic multifunction $\mathbf{g}_{\alpha}: \mathbb{D} \to \mathbb{D}.$

 The values of the extended $\mathbf{g}_{\alpha}$ at every singularity $x=x_j$ are determined by Theorem \ref{lkhlkjhsiedf}.  According to this result, every  branch of $\mathbf{h}_{\alpha}$ locally defined at $x_j+r$ can be extended to a punctured neighborhood of $x_j$, with a removable singularity at $x_j$, thus satisfying the conditions of \textcolor{black}{Lemma \ref{$(e1)$}} with $y=y_j.$  Hence the set $\mathbf{g}_{\alpha}(x_j)$ is given by all 
 \[h(x_j)=y_j \in \mathbf{g}(x_j), \]
  \noindent such that $(h, x_j +r)$ is a germ of $\mathbf{h}_{\alpha}.$ The last statement of Lemma \ref{$(e)$} follows from this observation. 
 
 The proof of Lemma \ref{$(e)$}  is complete. 
\end{proof}

\begin{lem} \label{$(f)$}
The family of global analytic multifunctions $\mathbf{g}_{\alpha}$ satisfies \eqref{fdfsdfsdfsdf}. 
\end{lem}

 \begin{proof} The inclusion $\cup_{\alpha} \mathfrak{S}(\mathbf{g}_{\alpha}) \subset \mathfrak{S}(\mathbf{g})$ follows from Lemma \ref{$(e)$}, for then every branch of $\mathbf{g}_{\alpha}$ is a branch of $\mathbf{g}.$  For the other inclusion, let $g$ be  a local branch of $\mathbf{g}$ defined on the open ball  $D$ with small radius $r>0$ and centered at a branch point $x_j.$ Let $g^*$ denote the restriction of $g$ to the punctured disk $D^*= D{\setminus} \{x_j\}.$ Then $g^*$ is a branch of some $\mathbf{h}_{\alpha}.$ By Theorem \ref{lkhlkjhsiedf}, $g^*$ coincides with some $F\in \mathcal{B}_j$  on $D^*$, which is well known to have  a removable singularity at $x_j$, being  a branch of $\mathbf{g}_{\alpha}$ defined on $D.$ Hence every local branch of $\mathbf{g}$ at $x_j$ is indeed a local branch of some $\mathbf{g}_{\alpha}.$ The first equality in \eqref{fdfsdfsdfsdf}  is satisfied. The second follows  from the first.
 \end{proof}

\begin{lem} \label{$(g)$}   Every germ $\mathfrak{f}$ of $\mathbf{g}_{\alpha}$ can be continued along any curve $\gamma: [a ,b] \to \mathbb{D}$   starting at $\pi(\mathfrak{f}).$

\end{lem}

 \begin{proof}
 We need to determine a continuous function $\tilde{\gamma}: [0, 1] \to \mathfrak{S}(\mathbf{g}_{\alpha})$ such that $\pi(\tilde{\gamma}(t)) = \gamma(t)$ on $[0, 1]$ and $\tilde{\gamma}(0)=\mathfrak{f}.$ 
The germ $\mathfrak{f}$ is determined by a local branch $f$ at $x=\gamma(0).$ Let $y=f(x).$ In a preliminary case, the curve $\gamma$ does not intersect the set of branch points $B$; then the continuation exists  by Lemma \ref{$(d)$}, for then $\mathbf{h}_{\alpha}$ coincides with $\mathbf{g}_{\alpha}$ on $\mathbb{D} {\setminus} B.$  

 In a second case, the only branch points in $\gamma$ are the starting and terminal points. It is easy to determine a continuation $\tilde{\gamma}(t)$ for $t$ sufficiently close to $0.$ Since $\gamma$ contains no branch points on $(0, 1)$, there exists a unique  extension of  $\tilde{\gamma}$ to a continuation along $\gamma|_{[0,t]}$ for $t$ arbitrarily close to $1.$ To define the value of the continuation near $1$ we apply Theorem  \ref{lkhlkjhsiedf}: if $t^*$ is sufficiently close to $1$, then there exists a unique branch $F$ of $\mathbf{g}_{\alpha}$ locally defined at $\gamma(1)$ such that $(F, \gamma(t^*)) = \tilde{\gamma}(t^*).$ Since any two analytic continuations are either identical or else differ for every $t,$ we have $(F, \gamma(t)) = \tilde{\gamma}(t)$ as long as $t<1$ and $\gamma(t)$ is in the domain of $F.$ If we set $\tilde{\gamma}(1)=(F, \gamma(1))$, the result is a continuous curve $\tilde{\gamma}$ in the sheaf; in other words, an analytic continuation along $\gamma.$

 In the general case, the compact set $\gamma([0,1])$ contains finitely many branch points at $\gamma(t_i)$ for $t_1 < t_2 < \cdots < t_n.$ By working with the restriction to each consecutive interval $[t_i, t_{i+1}]$ we reduce the analysis to the previous case, obtaining an analytic continuation along the whole curve. 
 \end{proof}

 This completes the proof of Theorem~\ref{gdsadgwed}, as each conclusion of the statement has been established through the preceding lemmas.

\section{Proof of Theorem \ref{ljksoijdfkwdfsed}}

\begin{proof}[Proof of Theorem \ref{ljksoijdfkwdfsed}]\label{adfaoihpoiapoiasudfasdf} By definition, every branch is holomorphic and defined on a region, that is, a nonempty, open and connected subset of the plane (see page \pageref{lkoikdkdkciied}). Every branch $f$ of $\mathbf{g}$ determines a connected set

 $$\Lambda = \{ (f,z) : z\in \operatorname{dom}(f)\}$$

\noindent which is contained in some component of the sheaf. By Theorem \ref{gdsadgwed}, this component must be some $\mathfrak{S}(\mathbf{g}_{\alpha})$. Hence $f$ is a branch of $\mathbf{g}_{\alpha}.$ From Theorem \ref{gdsadgwed} we know that  every germ $\mathfrak{f}$ in $\mathfrak{S}(\mathbf{g}_{\alpha})$ can be continued along any curve in $\mathbb{D}$ starting at $\pi(\mathfrak{f}).$ It follows from  Corollary \ref{kjhlkjhlskjed} that $f$ can be extended to a branch $F$ defined on $\mathbb{D}.$ 

Using Diagrams \eqref{sdfsdfskkdfdw} and \eqref{sdfsdfskkdfdwlkjlksded} we can see that if $x\in\mathbb{D}$ projects to the ramified point $a_j=0$, in which case the ramification index is $\nu(a_j)=q$, then $F'(x)=0$ and $F$ is locally conjugate to $z^q$ at $x$ (up to a translation), as a consequence of B\"ottcher's Theorem. Indeed, $\nu(0)=q$ implies $\varphi'(x)=0$ with local degree equal to $q.$ Besides, every local branch of $\mathbf{f}_a^{-1}$ is univalent on a neighborhood of $0,$ for then $0$ is not a singular point of $\mathbf{f}_{a}^{-1},$ because $a$ is Misiurewicz. Since the point $b$ of Diagram \eqref{sdfsdfskkdfdw} is a pre-image of $0$ under $\mathbf{f}_a^{-1}$, $b$ is not a ramified point (recall Definition \ref{lkjlhdoiehfeddd}), otherwise $0$ would be contained in a cycle, which is impossible, since $a$ is Misiurewicz. Hence $\nu(b)=1$ and some neighborhood $V$ of $b$ is evenly covered by $\varphi$, in the sense that $\varphi$ projects every component of $\varphi^{-1}(V)$ onto $V$ by means of a bi--holomorphic map (recall that $\varphi$ is a regular branched covering and this property holds at every point which is not ramified). We conclude that for every $z$ in a neighborhood $W$ of $x,$ 
\[F(z) = \psi \circ f^{-1} \circ \varphi(z), \]
\noindent where $f^{-1}$ is a univalent branch of $\mathbf{f}_a^{-1}$, $\psi$ is a univalent branch of $\varphi^{-1}$ and $\varphi'(x)=0$, with local degree equal to $\nu(0)=q.$ Since $F'(x)=0,$ $F$ is not injective on any neighborhood of $x$, and therefore cannot be a local isometry at $x.$ By the Schwarz-Pick Lemma (the version on Riemann surfaces), $F$ is a strict contraction of the hyperbolic metric.  
\end{proof}

\section{Uniform contraction on compact sets}\label{asdfasdfohasdfasd}

Recall from  Definition \ref{lkjlsjodijpoiefd} that the orbifold metric  of the canonical orbifold associated to a Misiurewicz point $a$ is a conformal metric defined on the complement $\mathbb{C}{\setminus} R$ of the locally finite set of ramified points $R$.

\begin{thm} \label{kljhlsdiedf} Suppose that $a$ is a Misiurewicz point for the family \eqref{lkjdllkjhalsoed}. Let $\rho$ denote the orbifold metric of $\mathbb{C}{\setminus}R$, where $R$ denotes the set of ramified points $a_j$ of the canonical orbifold.  For every compact  $K \subset \mathbb{C}$  such that $K{\setminus} R$ is nonempty, there exist $\eta_{K}$ in $(0,1)$ and $r>0$ such that any local branch of $\mathbf{f}_{a}^{-1}$ defined on a small region containing $\zeta_0$ in $ K{\setminus} R$ can be extended to a univalent  branch $g$ of $\mathbf{f}_a^{-1}$ which is defined on the open ball $B(\zeta_0, r_0),$ where 
\[r_0=\min\{r, d(\zeta_0, R) \}.\]
\noindent Moreover,   $\|g'(w)\|_{\rho} < \eta_{K},$ for every $w$ in $B(\zeta_0, r_0).$ 
\end{thm}

 The proof will be given after some lemmas.

A  subset $\mathcal{D}$ of the plane is a \emph{regular disk} if $\mathcal{D}=g(D)$, where $D$ is an open ball and $g$ is a univalent map defined on a region containing $\overline{D}.$

\begin{lem}\label{lkhalosiddwesdfc} Let $\mathcal{D}$ be a regular disk in $\mathbb{C}.$ Suppose that  $\mathcal{D}$ contains a unique ramified point $a_j.$  Let $\varphi$ be the regular branched covering of Theorem \ref{gdsadgwed}.
 Given $x_j$ in the pre-image of $a_j$ under $\varphi$,   let $\tilde{\mathcal{D}}_j$ be the unique connected component of  $\varphi^{-1}(\mathcal{D})$ which contains $x_j.$ Then the closure of $\tilde{\mathcal{D}}_j$ is contained in $\mathbb{D}$ and
 $\varphi: \tilde{\mathcal{D}}_j \to \mathcal{D}$
  \noindent is a proper map of degree $\nu_a(a_j),$ with only one branch point at $x_j.$ In particular, $\varphi$ restricts to a $\nu_a(a_j)$-fold covering map from $\tilde{\mathcal{D}}_j{\setminus}\{x_j\}$ onto $\mathcal{D}{\setminus} \{a_j\}.$

\end{lem}

\begin{proof}  There exists a conformal disk $V$ containing $a_j$  such that every connected component of $\varphi^{-1}(V)$
projects onto $V$ by means of a proper map. The group of deck transformations acts transitively on such components. Since the set \textcolor{black}{of} branch points is locally finite, by reducing $V$ if necessary we may assume that each component contains only one branch point, which must be projected to $a_j$ by $\varphi.$ Let $\mathcal{E}$ denote the unique component of $\varphi^{-1}(V)$ containing $x_j.$ Then $\varphi: \mathcal{E}{\setminus} \{x_j\} \to V{\setminus} \{a_j\}$ is a  $d$-fold covering map, where $d=\nu_a(a_j).$

The set $\tilde{\mathcal{D}}_j$ in the statement is constructed as follows. Fix any point $z_0$ in $V{\setminus} \{a_j\}$. Any $z$ in $\mathcal{D} {\setminus} \{a_j\}$ can be joined to $z_0$ by a continuous curve $\gamma$ in $\mathcal{D} {\setminus} \{a_j\}$ starting at $z_0.$ The analytic continuation of a germ of $(\varphi|_{\mathcal{E}})^{-1}$ at $z_0$ can be continued along $\gamma$ and leads to a germ at $z$ which is represented by $(g_1, z)$, where $g_1$ is a branch of $\varphi^{-1}$ locally defined at $z.$ For every $z$ in $\mathcal{D}{\setminus} \{a_j\},$ we let $\mathbf{h}(z)$ denote the set of all $g_1(z)$ such that $(g_1, z)$ is obtained by analytic continuation of a germ of $(\varphi|_{\mathcal{E}})^{-1}$ at $z_0$ along a curve $\gamma$ in $\mathcal{D}{\setminus} \{a_j\}$ joining $z_0$ to $z$, as described previously. We set $\mathbf{h}(a_j)=x_j.$ By the Monodromy Theorem, $\mathbf{h}(z)$ does not depend on  the initial choice of $z_0$ and $\mathbf{h}(V)=\mathcal{E}.$ The set  $\mathcal{C}=\mathbf{h}(\mathcal{D})$ -- which is constructed using analytic continuations along curves -- is naturally path connected and satisfies
\begin{equation}\label{lkjsldoioiesfsdf} \varphi(\partial \mathcal{C}) \subset \mathbb{C}{\setminus}\mathcal{D}.
\end{equation}
In order to show that $\overline{\mathcal{C}} \subset \mathbb{D}$, assume that $\mathcal{D}=g(D)$, where $g$ is a univalent map defined on a region containing the closure of a round disk $D.$ For a small $\epsilon>0$, the $\epsilon$-neighborhood of $D$ is contained in the domain of $g.$ Let $\mathcal{D}_{\epsilon} = g(D_{\epsilon}).$ 
Using the same process in the construction of $\mathcal{C},$  we construct another region $\mathcal{C}_1$ containing the closure of $\mathcal{C}$ by considering all possible analytic continuations of germs of $(\varphi|_{\mathcal{E}})^{-1}$  along curves in $\mathcal{D}_{\epsilon}$. It should be noticed, however, that now $\mathcal{D}_{\epsilon}$ might contain a finite set of ramified points, and as a consequence, some analytic continuations do not lead to a germ, but rather to a branch point of $\varphi$ which we add to $\mathcal{C}_1$ in the process of its construction. It follows that $\overline{\mathcal{C}}\subset \mathcal{C}_1 \subset \mathbb{D}.$ 
\textcolor{black}{If $K$ is a compact subset of $\mathcal{D}$},  then by \eqref{lkjsldoioiesfsdf} the intersection $\varphi^{-1}(K) \cap \partial \mathcal{C}$ is empty. Hence 
\[(\varphi|_{\mathcal{C}})^{-1}(K) = \varphi^{-1}(K) \cap \mathcal{C} = \varphi^{-1}(K) \cap \overline{\mathcal{C}} \]

\noindent is compact. Since $K$ is arbitrary, $\varphi: \mathcal{C} \to \mathcal{D}$ is proper, with the same degree $d$ of $\varphi|_{\mathcal{E}}.$ Finally, we will prove that $\mathcal{C}=\tilde{\mathcal{D}}_j.$ Since $\mathcal{C}$ is connected and $x_j\in \mathcal{C}$, it follows that $\mathcal{C} \subset \tilde{\mathcal{D}}_j.$ For the other inclusion, we consider a curve $\zeta$ in $\tilde{\mathcal{D}}_j$ joining some $w_0 \in \mathcal{E}$ to an arbitrary $w$ in $\tilde{\mathcal{D}}_j.$ Without loss of generality,
we may assume that $\gamma =\varphi\circ \zeta$ contains no ramified points;  by Theorem \ref{jdjfsfdwef}, $\zeta$ is given by an analytic continuation  of a germ of $(\varphi|_{\mathcal{E}})^{-1}$ along $\gamma$, thereby showing that $w$ belongs to $\mathcal{C}.$

The proof follows with $\tilde{\mathcal{D}}_j = \mathcal{C}.$
\end{proof}

\begin{lem}\label{ojpoisdfweadafd}  Let $s_j$ be the maximal positive real number such that the open ball $B(a_j, s_j)$ does not intersect $R,$ except for $a_j.$ Then there exists $\eta_j$ in $(0, 1)$ such that any  branch $g$ of $\mathbf{f}_a^{-1}$ defined on a small connected neighborhood of some $\zeta_0\in B(a_j, s_j/2){\setminus \{a_j\}}$ can be extended to a univalent branch of $\mathbf{f}_a^{-1}$ defined on $B(\zeta_0, r_0)$, where $r_0$ denotes the Euclidean distance between $\zeta_0$ and $R.$ Moreover, $g$ strictly contracts the orbifold metric $\rho$ by the uniform factor $\eta_j:$
\[\|g'(z)\|_{\rho} < \eta_j, \]
\noindent for every $z$ in $B(\zeta_0, r_0).$

\end{lem}

\begin{proof} Let $\mathcal{D}_j$ denote the maximal disk $B(a_j, s_j).$ Fix an arbitrary $x_j$ in the pre-image of each $a_j$ under $\varphi.$ Using the same terminology of Lemma \ref{lkhalosiddwesdfc}, let $\tilde{\mathcal{D}}_j$ be the unique connected component of  $\varphi^{-1}(\mathcal{D}_j)$  which contains $x_j.$  Note that the closure of each $\tilde{\mathcal{D}}_j$ is a compact subset of $\mathbb{D}.$
We may assume $a_0=0$ is the first ramified point and $a_1=a$ is the second. Since $a$ is the singular point of $\mathbf{f}_a^{-1}$, it maps  $\mathcal{D}_1$ onto the  ball $B(0, s_1^{q/p})$, which is contained in $\mathcal{D}_0$, since  $\mathbb{C}{\setminus}R$ is backward invariant under $\mathbf{f}_a$ and $s_0$ is maximal.

Given a univalent branch $g$ of $\mathbf{f}_a^{-1}$ defined on a conformal disk $V_1\subset \mathcal{D}_1{\setminus} \{a\},$ we may use the Lifting Theorem of covering spaces, Lemma \ref{lkhalosiddwesdfc} and the fact that $\mathbf{f}_a^{-1}$ is a separable holomorphic multifunction on $\mathbb{C}{\setminus}\{a\}$ to construct
 a commutative diagram of bi-holomorphic maps between  conformal disks $V_0 \subset \mathcal{D}_0{\setminus}\{0\}$, $\tilde{V}_0 \subset \tilde{\mathcal{D}}_0{\setminus}\{x_0\}$ and $\tilde{V}_1 \subset \tilde{\mathcal{D}}_1{\setminus} \{a \}:$

\begin{equation}\label{sdfsdfskkdfdfasded}
    \begin{tikzcd}
      \tilde{V}_1   \arrow{r}{G}\arrow{d}[swap]{\varphi} & \tilde{V}_0 \arrow{d}{\varphi} \\
        V_1 \arrow{r}[swap]{g} & V_0
    \end{tikzcd}
\end{equation}

By Theorem \ref{ljksoijdfkwdfsed} and Lemma \ref{lkhalosiddwesdfc},    $G$ can be extended to a global branch $G: \mathbb{D} \to \mathbb{D}$, with $G(\tilde{\mathcal{D}}_1) \subset \tilde{\mathcal{D}}_0$, which is a strict contraction of the hyperbolic metric $\mu$ on $\mathbb{D}.$  
Since $\nu(0)=q$ and $\# \mathbf{f}_a^{-1}(z) =p$ if $z\neq a$, it is possible to use Lemma \ref{lkhalosiddwesdfc} and Diagram \eqref{sdfsdfskkdfdfasded} to show that the space $\mathcal{G}_1$ consisting of all $G$ obtained by the previous method, with $G(\tilde{\mathcal{D}}_1) \subset \tilde{\mathcal{D}}_0$, has cardinality $\#  \mathcal{G}_1 = pq.$ Since each member  of $\mathcal{G}_1$ strictly contracts $\mu$ on the relatively compact set $\tilde{\mathcal{D}}_{1}$, there exists $\eta_1 \in (0, 1)$ such that 

\begin{equation}
\| G'(w)\|_{\mu} < \eta_1
\end{equation}
for every $G$ in $\mathcal{G}_1$ and every $w\in {\tilde{\mathcal{D}}}_1.$
In Diagram \eqref{sdfsdfskkdfdfasded}, $\varphi$ is an isometry on each disk (using the hyperbolic metric $\mu$ on the unit disk and the orbifold metric $\rho$ on $\mathbb{C}{\setminus} R$).  We conclude that $\|g'(z)\|_{\rho} < \eta_1,$ whenever $z\in V_1$, for every univalent branch $g$ defined on $V_1.$ We may take $V_1$ as the open ball $B(\zeta_0, r_0)$ described in the statement. The Lemma follows in the case $a_j=a.$  The general case involves a similar reasoning  (using another sequence of univalent maps just like  Diagram \eqref{sdfsdfskkdfdfasded}, with the help of Lemma \ref{lkhalosiddwesdfc}).

\end{proof}

\begin{proof}[Proof of Theorem \ref{kljhlsdiedf}] Follows from 
Theorem \ref{ljksoijdfkwdfsed} and  Lemma \ref{ojpoisdfweadafd}, using a standard compactness argument, with a finite covering by relatively compact open sets on which all (finitely many) branches of $\mathbf{f}_a^{-1}$ contract the orbifold metric by a uniform factor. Lemma \ref{ojpoisdfweadafd} is applied to find this uniform factor on a neighborhood of every ramified point.  Using the polar representation of $\mathbf{f}_a^{-1}$ and the fact that $p/q>1$ it is possible to show that $g$ extends to a univalent branch on $B(\zeta_0, r_0)$ by computing $g$ explicitly using polar coordinates. 
\end{proof}

\section{Subhyperbolicity} \label{asdfapoiuasdfasdrweg}

If $a$ is a Misiurewicz point, then by definition the critical point has a unique bounded forward orbit, the \emph{preperiodic critical orbit of $\mathbf{f}_a$.} This critical orbit is precisely the set of ramified points which are in $K_a$. (Recall that every ramified point is a singularity of the orbifold metric). Since both $K_a$ and $\mathbb{C}{\setminus}R$ are backward invariant, it follows that $K_{a}{\setminus}R$ is also backward invariant under $\mathbf{f}_a.$ The following result guarantees the uniform expansion of the orbifold metric within a neighborhood of $K_a.$

\begin{thm}[\bf Subhyperbolicity] \label{poiaudfposiuwedfs} Suppose that $a$ is a Misiurewicz point for the family \eqref{lkjdllkjhalsoed}. Then $\mathbf{f}_a$ expands the orbifold metric $\rho$ on a neighborhood of $K_a$ by a uniform factor. More precisely, there exist an open set $V$ containing $K_a$ and a constant $\eta\in (0, 1)$ such that $R\cap V$ is the preperiodic critical orbit of $\mathbf{f}_a$, and for every univalent branch $g$ of $\mathbf{f}_a^{-1}$ defined on a region $W\subset V,$ we have $\| g'(w)\|_{\rho} <\eta$, for every $w\in W{\setminus}R.$
\end{thm}

\begin{proof}[Proof of Theorem \ref{poiaudfposiuwedfs}]  

Since the set $R$ is locally finite and $K_a$ is compact, there exists an $\epsilon$-neighborhood  $V=(K_a)_{\epsilon}$ such that  $R\cap V =R\cap K_a.$ 
 
 \noindent By Theorem \ref{kljhlsdiedf}, there exists a uniform contracting factor $\eta \in (0, 1)$ associated to the compact set  $\overline{V}$; if $g$ is a univalent branch of $\mathbf{f}_a^{-1}$ defined on a region $W \subset V,$  then $\|g'(\zeta_0)\|_{\rho} < \eta$, for every  $\zeta_0$ in $W{\setminus} R.$
\end{proof}

\gap \paragraph{Change of variables.} We may regard $\mathbb{C}$ as a Riemann surface on which every univalent  map  $\phi$ defined on a region $\Omega \subset \mathbb{C}$ is a coordinate chart. It is usual to call $z=\phi(\zeta)$ a local uniformizing parameter in $\phi(\Omega).$ The push-forward of a conformal metric $\rho$ defined on $\Omega$ is a metric on $\phi(\Omega)$ that turns $\phi$ into an isometry onto its image; it is often referred to as the \emph{expression of  $\rho$ with respect to the local uniformizing parameter $z$.}

\begin{lem}\label{fsdfaqefasdfasd}For every ramified point $a_j$ of the canonical orbifold associated to a Misiurewicz point there exists a coordinate chart $z=\phi(\zeta)$ defined for $\zeta$ in a neighborhood of $a_j$, with $\phi(a_j) =0$, such that the expression of the orbifold metric with respect to the local uniformizing parameter $z$ around zero becomes 
\begin{equation}\label{lksdfwesdfs} ds = \frac{\rho_j(\sqrt[d]{z})}{|z|^{1-1/d}} |dz|\end{equation}
\noindent where $d=\nu(a_j)$ and  $\rho_j$ is a $C^\infty$ and strictly positive function defined on a round disk centered at $0$  which satisfies  $\rho_j(e^{2\pi i/d}z) =\rho_j(z),$ for any $z$ in the domain of $\rho_j.$ 
\end{lem}

\begin{proof}[Reference to the proof] The statement can be generalized for any regular branched covering $\varphi$ (not necessarily the one used in this paper). Let $x_j$ be a branch point.  Up to compositions with translations,  $\varphi$ is locally conjugate to $z^{\nu(a_j)}$, where $a_j=\varphi(x_j)$, as a consequence of B\"ottcher's Theorem.   \textcolor{black}{The expression \eqref{lksdfwesdfs} is found in \cite[p. 211]{Milnor},} using the branched covering $z^{\nu(a_j)}$ around zero,  computing directly the push-forward of the metric. \end{proof}

  Recall that in the case of a Misiurewicz point,  the preperiodic critical orbit of $\mathbf{f}_a$ contains a unique cycle $\alpha(a).$

 \begin{thm}\label{lhlkjhapopoxooxdfs} If $a$ is a Misiurewicz point, then  $\alpha(a)$ is a repelling cycle contained in the filled Julia set $K_a.$ 
 \end{thm}

\begin{proof} Since every cycle is a bounded orbit, it lies in the filled Julia set. The points of $\alpha(a)$ are in $R$ and none of them is the critical point. Denote the points of $\alpha(a)$ by 

\begin{equation}\label{fadfqwefassdwedffd} \zeta_0 \mapsto \zeta_1 \mapsto \zeta_2 \mapsto \cdots \mapsto \zeta_n=\zeta_0.
\end{equation}

\noindent There exist univalent branches $f_i$ of $\mathbf{f}_a$, each of which is locally defined at $\zeta_{i-1}$, satisfying  $f_i(\zeta_{i-1}) = \zeta_i$. Consider the composition 
\begin{equation} \label{sdfasdewefcx}f= f_n \circ f_{n-1} \circ \cdots \circ f_1, \end{equation}

\noindent  which has a fixed point at $\zeta_0$. Let $\phi$ be a local chart at $\zeta_0$ with $\phi(\zeta_0) = 0.$ Then $g=\phi \circ f \circ \phi^{-1}$ has a fixed point at $z=0$ and $g'(0)$ coincides with the multiplier of $\alpha(a)$ (recall that the multiplier is invariant under conformal conjugacies).  It suffices  to show that $|g'(0)| > 1.$ We know that $g'(0) \neq 0,$ for then the cycle $\alpha(a)$ does not contain zero. As  usual, let $\rho$ denote the orbifold metric on $\mathbb{C}{\setminus} R.$ By Theorem \ref{poiaudfposiuwedfs},  each $f_i$ expands $\rho$ by a uniform factor $\lambda >1$ on a punctured neighborhood of  $\zeta_i.$ It follows from the chain rule that $\|f'(\zeta)\|_{\rho} \geq \lambda^n$, for every $\zeta\neq \zeta_0$ in a neighborhood of $\zeta_0.$

Since $\zeta_0$ is a ramified point $a_j$, using Lemma \ref{fsdfaqefasdfasd} we determine a local expression of $\rho$ with respect to the local uniformizing parameter $z=\phi(\zeta)$ around zero, just like  \eqref{lksdfwesdfs}. Let  $d=\nu(a_j)$ and $\ell =1-1/d.$ By \eqref{lksdfwesdfs}, for every $z=\varphi(\zeta)$ in a punctured neighborhood of zero, 
\[\lambda^n \leq  \|f'(\zeta)\|_{\rho} = \|g'(z) \|_{\rho_j} = |g'(z)| \frac{\rho_j(\sqrt[d]{g(z)})}{\rho_j(\sqrt[d]{z})}\frac{|z|^{\ell}}{|g(z)|^{\ell}}.\]

\noindent Taking the limit as $z\to 0,$ the last product converges to $|g'(0)|/|g'(0)|^{\ell}$ because $\rho_j(0)>0.$ Hence $|g'(0)|^{1/d}\geq \lambda^n$, from which we conclude that $\alpha(a)$ is repelling. 
\end{proof}

 For the basic properties of $J_c$ and $K_c$, see Definition \ref{adfadweerdccgqwe}.

\begin{thm}\label{fsdfawedsfasdeed} Suppose that $a$ is a Misiurewicz point for the family \eqref{lkjdllkjhalsoed}. Then $J_a = K_a$ and 
 \begin{equation}\label{ljhsldfswdfsd}
K_a = \overline{\bigcup_{n\geq 0} \mathbf{f}_a^{-n
}(0)}.\end{equation}
\end{thm}

\begin{proof} Let $K=K_a.$
Recall that the points of the preperiodic orbit of zero coincide with $R\cap K,$ where $R$ is the set of ramified points. Let $R^*= R {\setminus} \{0\}.$ We know that every point $z$ of $K\cap R^*$ has only one image in $K$  (otherwise there would be another bounded forward orbit of zero), and the other $q-1$ images of $z$ are outside the compact set $K.$ By continuity, the correspondence maps a small conformal disk containing $z\in K \cap R^*$ to $q$ disjoint conformal disks, and only one of them intersects $K.$ This determines, at each such $z$, a \textcolor{black}{special} univalent branch $g_z$ of the correspondence that sends a conformal disk containing $z$ to the aforementioned conformal disk intersecting $K.$ The disjoint union of  such small disks determines a neighborhood $V_{\alpha}$ of $K\cap R^*$ such that, if the correspondence sends a point $\zeta \in V_{\alpha}$ to another point $\zeta'$ in $V_{\alpha}$, then $\zeta'=g_z(\zeta)$, for some $z\in K\cap R^*.$  This creates on $V_{\alpha}$ a \textcolor{black}{special} regime of iteration; and since $\alpha(a)$ is repelling, any infinite forward orbit contained in $V_{\alpha}$ is preperiodic and eventually coincides with $\alpha(a).$

 We will need the following lemma. 
\begin{lem}\label{fsdfsdwedf}
For any conformal disk $U$ intersecting $K_a$, there exists a finite sequence of nonconstant holomorphic maps $(f_j)_{j=0}^n$, defined on $U$ and {\color{black}starting with the identity map $f_0=\mathrm{Id}$, such that each $f_j$ is a branch of the composition $\mathbf{f}_a\circ f_{j-1}$,} $f_n(U)$ contains the origin, and $f_j'(z)\neq 0$ for all $z\in U$ and all $j\in\{0,\dots,n\}$. Since $U$ is arbitrary, \eqref{ljhsldfswdfsd} follows.

\end{lem}

 \begin{proof}[Proof of Lemma \ref{fsdfsdwedf}] Fix a bounded forward orbit $(z_i)_0^{\infty}$ of a point in $U\cap K_a.$ If $0\in U$ or some point of this orbit is a critical point, then there is nothing to prove. Otherwise, by the Monodromy Theorem and its Corollary \ref{kjhlkjhlskjed} there exists a unique branch $f_1$ of $\mathbf{f}_a$ defined on $U$ which sends $z_0$ to $z_1$. If $f_1(U)$ contains $0$, then there is nothing to prove. Otherwise, we may consider the holomorphic multifunction $\mathbf{f}_a \circ f_1$ defined on $U$, and using  analytic continuation as before, we find a unique branch $f_2$ of $\mathbf{f}_a \circ f_1$ defined on $U$ which sends $z_0$ to $z_2.$ Unless $f_2(U)$ contains the critical point,  $f_3$ is by definition the unique branch of $\mathbf{f}_a \circ f_2$ defined on $U$ which sends $z_0$ to $z_3.$ If we are fortunate,  this argument terminates at a branch $f_n$ of $\mathbf{f}_a^n$ defined over $U$ such that $f_n(U)$ contains zero, and there is nothing else to prove. Otherwise, we find a sequence of maps $f_n: U \to \mathbb{C}{\setminus} \{0, a\}$, which is normal, by Montel's Theorem. However, we will see that $f_n$ can never be a normal family. Suppose first that the set of subsequential  limits of $(z_i)_0^{\infty}$ is contained in $R\cap K.$  Then all but finitely many $z_n$ belong the the neighborhood $V_{\alpha}$ of $K\cap R^*$. Hence $(z_i)_0^{\infty}$ eventually coincides with the repelling cycle $\alpha,$ which implies $f_n'(z_0) \to \infty$, as a consequence of the chain rule. We conclude from the Weierstrass convergence theorem that  $f_n$ is not normal.
If the set of subsequential limits of the bounded orbit $(z_i)_0^\infty$ is not contained in $K\cap R,$  then a subsequence $z_{n_k}$ converges to some $w_0$ in $K{\setminus}R.$  Suppose for a moment that $f_n$ is a normal family. Since $f_{n_k}(z_0)$ converges to $w_0,$ no subsequence of $f_{n_k}$ escapes to infinity. By normality, we may  replace $f_{n_k}$ by one of its convergent subsequences, so that  $f_{n_{k}}$  converges locally uniformly to some holomorphic function $g$ defined on $U$, as well as $f'_{n_{k}}$ converges locally uniformly to $g'.$ Since $R$ is forward invariant and $K\cap R$ is the unique bounded forward orbit of zero,  it is possible to show that if  one point of the sequence $(z_i)_0^\infty$ enters $R$, then subsequent terms never leave $R$ and are eventually trapped in the repelling cycle $\alpha(a)$, a possibility that is ruled out by the fact $z_{n_k} \to w_0 \notin R.$ Hence we are allowed to evaluate the norm of $f_{n_k}'$ with respect to the orbifold metric $\rho$, concluding from 
Theorem \ref{poiaudfposiuwedfs} and the chain rule that $\|f_{n_k}'\|_{\rho}$ explodes to infinity as $k\to \infty$; nevertheless, we also have that

$$\|f_{n_k}'(z_0) \|_{\rho} = |f_{n_k}'(z_0) | \frac{\rho(f_{n_k} (z_0))}{\rho(z_0)} \to |g'(z_0)| \frac{\rho(w_0)}{\rho(z_0)}$$ 

\noindent which is evidently a contradiction. Hence $f_n$ is never normal, and some iterate $f_n(U)$ contains the critical point.  \end{proof}

  As described in \eqref{fadfqwefassdwedffd} and \eqref{sdfasdewefcx}, the composition of the univalent branches along the points $\zeta_j$ of the repelling cycle $\alpha$ yields a holomorphic  map $f$ defined on a neighborhood of $\zeta_0$ with a repelling fixed point at $\zeta_0=f(\zeta_0).$ The inverse $f^{-1}$ maps a conformal disk $\Omega$ around $\zeta_0$ into itself, with a geometrically attracting fixed point of multiplier $\lambda_0$ at $\zeta_0.$ We may choose $\Omega$ sufficiently small so that $\zeta_0$ is the unique ramified point contained in $\Omega$.
Let $\mathbb{D}_r$ denote the open disk of radius $r$ centered at zero.  Using  the K\"onigs linearization theorem, we may assume that $\Omega=\varphi(\mathbb{D}_r),$ where $\varphi$ is a conformal conjugacy between  $z\mapsto \lambda_0z$ and $f^{-1}$. 

Let $\tilde{z}_0$ be any point of $\Omega{\setminus} \{\zeta_0\}.$ Successive iterations of this point with respect to the maps $f_i^{-1}$ described in \eqref{sdfasdewefcx} determine an orbit $(\tilde{z}_i)_0^{\infty}$ of $\mathbf{f}_a^{-1}$ converging to the cycle $\alpha$ (in the backward direction), in the sense that  the $\omega$-limit set of this orbit is the union of all  $\{\zeta_i\}.$

{\color{black}
Let $U$ be an arbitrary conformal disk intersecting $K_a.$
By Lemma \ref{fsdfsdwedf}, a holomorphic branch $H$ of some iterate of the correspondence $\mathbf{f}_a$ sends $U$ to a connected neighborhood $V$ of zero, with $H'(z)\neq 0$ for all $z\in U.$   Thus $H(w_0)=0$ for some $w_0\in U$. Notice that $w_0 \not \in R$, since $a$ is Misiurewicz. Moreover, $W=\mathbf{f}_a(V)$ is a connected neighborhood of $a$.

We may  choose a simply connected region $U_1$ containing $w_0$ and $\tilde{z}_0$ such that $\overline{U_1}$ is contained in the complement of the set of ramified points $R.$}

There exists a conformal map $h$ defined on a neighborhood of $a$ which is a branch of some iterate $\mathbf{f}_a^{k_0}$ and sends $a$ to $\zeta_0.$ By choosing $s \in (0, r)$ sufficiently small, we may assume that $h$ maps a neighborhood $W_1$ of $a$ contained in $W$ biholomorphically onto $\Omega_1=\varphi(\mathbb{D}_{s}) \subset \Omega.$   Let $V_1$ denote $\mathbf{f}_a^{-1}(W_1)$, which is   a connected neighborhood of zero contained in $V.$ If we reduce the size of $s$, then both sets $V_1$ and $W_1$ will shrink accordingly; since $H$ is conformal and sends $w_0$ to zero, by reducing $s$ if necessary we may assume that $H$ sends a small connected neighborhood $U_2 \subset U_1$ of $w_0$ biholomorphically onto $V_1.$  It is important to notice that $U_2$ also depends on $s$, and if $s\to 0$, then $U_2$ will shrink to $\{w_0\}.$

Since $U_1 \subset \mathbb{C}{\setminus}R$, we are allowed to compute the length $\ell_{\rho}(\gamma)$ of curves in $U_1$ with respect to the orbifold metric $\rho$, as well as the diameter $|U_1|_{\rho}$ of $U_1$ with respect to $\rho.$  Choose some closed ball $\tilde{K}$ of large radius whose complement is a forward invariant set contained in the basin of infinity $\mathbb{C}{\setminus} K_a.$ Hence $\mathbf{f}_a^{-1}(\tilde{K})\subset \tilde{K}$ and we may assume that $K_a$ is contained in the interior of $\tilde{K}.$   Since the set of ramified points is locally finite, $\tilde{K}$ contains only finitely many points of $R;$  it is not hard to check that we can always take $U_1 \subset \tilde{K}$ satisfying the following property: there exists a positive constant $C_0$ such that any two points of $U_1$ are joined by a curve $\gamma$ within $U_1$ with $\ell_{\rho}({\gamma}) \leq C_0$ (indeed, it suffices to take $U_1$ as an $\epsilon$-neighborhood of some  curve joining $w_0$ and $\tilde{z}_0$ avoiding the finite set $\tilde{K}\cap R$).

Since $U_1$ is simply connected and the critical value $a$ is not in $U_1$, by analytic continuation there exists a unique branch $g_1$ of $\mathbf{f}_a^{-1}$ defined on $U_1$ which sends $\tilde{z}_0$ to $\tilde{z}_1$.  The set $\tilde{K}{\setminus} R$ is backward invariant under $\mathbf{f}_a$ and contains $U_1$; thus $g_1(U_1)$ is also contained in $\tilde{K}{\setminus} R.$ Inductively, we construct an infinite sequence of branches $g_n$ of $\mathbf{f}_{a}^{-n}$ such that $g_n$ is the unique branch of $\mathbf{f}_a^{-1}\circ g_{n-1}$ defined on $U_1$ which sends $\tilde{z}_0$ to $\tilde{z}_n.$

By the chain rule and Theorem \ref{kljhlsdiedf} we find $\eta \in (0, 1)$ such that $\|g_j'(z)\|_{\rho} < \eta^{j}$, whenever $z\in U_1$ and  $j > 0.$    If $g_j$ sends a pair of points $z, w$ in $U_1$ to another pair $x, y$ in $g_j(U_1)$ and $\gamma$ is a curve in $U_1$ joining $z$ to $w$ with $\ell_{\rho}(\gamma) \leq C_0$, then $$d_{\rho}(x, y) \leq \ell_{\rho}(g_j \circ \gamma) \leq \int_{0}^{1} \|g_j'(\gamma(t))\|_{\rho} \cdot \|\gamma'(t)\|_{\rho} dt \leq \eta^{j}C_0.   $$

\noindent Hence the diameter $|g_j(U_1)|_{\rho} \leq \eta^j C_0 $ tends to zero  as $j\to \infty.$ There exists a subsequence of $(\tilde{z}_j)_0^{\infty}$ converging to $\zeta_0,$  which is  the first point of the cycle $\alpha$. It is possible to show that some backward iterate of $U_1$ yields a small set $g_k(U_1)$ contained in $\Omega.$   Fix such $k$ and let $O_1 = g_k(U_1).$ 
By reducing $s$ if necessary, we may assume that $g_k$ maps $U_2$ biholomorphically onto a small subset $O_2$ of $O_1$ satisfying the following property: the closure of $O_2$ is contained in a conformal sector $\varphi(\mathbb{D}_r {\setminus} L)$ obtained by removing from $\mathbb{D}_r$ a slit connecting the origin to the boundary of $\mathbb{D}_r$ (recall that $\varphi$ conjugates $f^{-1}$ on $\Omega$ to $z\mapsto \lambda_0z$ on $\mathbb{D}_r$ and $\varphi(\mathbb{D}_r) = \Omega$).  
Some iterate $f^{-n_1}(O_2)$ is a small connected set $\Omega_2$ whose closure is also contained in a conformal sector $\varphi(\mathbb{D}_{r^{n_1}} {\setminus} L_1) \subset \Omega_1$, where $L_1$ is a slit obtained from a suitable rotation of $L.$  Let $\tilde{L}=\varphi(L_1).$ Then $\Omega_1{\setminus} \tilde{L} = \varphi(\mathbb{D}_s {\setminus} L_1)$ is also a conformal sector; in particular, it is a simply connected set excluding $\zeta_0.$ It follows that $h^{-1}$ sends $\Omega_1 {\setminus} \tilde{L}$ biholomorphically onto another conformal sector $W_1{\setminus} \tilde{L}_a,$ where $\tilde{L}_a = h^{-1}(\tilde{L})$ connects $a$ to the boundary of $W_1.$ Since $W_1{\setminus} \tilde{L}_a$ is simply connected and does not contain $a,$ it is the domain of a branch $F$ of $\mathbf{f}_a^{-1}$ whose image is contained in  $V_1.$ 

The whole idea of the proof is based on the following sequence of holomorphic maps between hyperbolic Riemann surfaces
\begin{equation}
\label{afsdfadfasewdsdfs}
\Omega_1{\setminus} \tilde{L}\xrightarrow{h^{-1}}  W_1 {\setminus} \tilde{L}_a   \xrightarrow{F}   V_1 \xrightarrow{(H|_{U_2})^{-1}} U_2 \xrightarrow{g_k} O_2 \xrightarrow{f^{-n_1}} \Omega_2  \Subset \Omega_1 {\setminus} \tilde{L}
   \end{equation}
which determine an attracting fixed point in $\Omega_2$, by the Schwarz-Pick Theorem. The fixed point is attracting for the composition of maps in \eqref{afsdfadfasewdsdfs}; therefore, it is a repelling periodic point of $\mathbf{f}_a$. We conclude that some repelling cycle intersects every set in \eqref{afsdfadfasewdsdfs}. In particular, $U$ contains a repelling periodic point of $\mathbf{f}_a$. Since $U$ is arbitrary, such points are dense in $K_a$, from which we conclude that $K_a=J_a.$

The proof of Theorem \ref{fsdfawedsfasdeed} is complete. \end{proof}

\gap \paragraph*{Acknowledgments}

\thanks{The author sincerely thanks Daniel Smania for his kind hospitality at ICMC/USP and Luna Lomonaco for valuable comments and insightful discussions. This work was partially supported by CNPq/MCTI/FNDCT under grant 406750/2021-1. The author is grateful to the anonymous referee for a careful reading of the manuscript and helpful suggestions that improved the presentation.}

\bibliography{oi}

\end{document}